\documentclass[12pt]{amsart}
\usepackage{amsfonts,amsmath,amssymb,amscd,amsthm,enumerate}
\usepackage{amsrefs}
\usepackage{geometry}
\usepackage{adjustbox}
\usepackage[all]{xy}
\usepackage{stmaryrd}
\newtheorem{theorem}{Theorem}
\newtheorem{prop}[theorem]{Proposition}
\newtheorem{lem}[theorem]{Lemma}
\newtheorem{cor}[theorem]{Corollary}
\theoremstyle{definition}

\newtheorem{rem}[theorem]{Remark}
\newtheorem{mydef}[theorem]{Definition}
\newtheorem{example}[theorem]{Example}

\renewcommand{\epsilon}{\varepsilon}
\def\lbra{{[}\!{[}}
\def\rbra{{]}\!{]}}
\def\ad{{\rm ad}}

\def\C{\mathbb{C}}

\def\R{\mathbb{R}}

\def\T{\mathbb{T}}

\def\x{\mathbf{x}}

\def\v{\mathbf{v}}
\def\u{\mathbf{u}}
\def\y{\mathbf{y}}
\def\z{\mathbf{z}}
\def\w{\mathbf{w}}
\def\e{\mathbf{e}}
\def\b{\mathbf{b}}
\def\<{\langle}
\def\>{\rangle}

\def\ad{{\rm ad}}

\newcommand{\BigWedge}{\mathord{\adjustbox{valign=B,totalheight=.6\baselineskip}{$\bigwedge$}}}

\usepackage{multicol} %%%%%%%%%%%

\begin{document}
\title{Polynomial Structures in Generalized Geometry}
\author{Marco Aldi and Daniele Grandini}
% Keywords command
\providecommand{\keywords}[1]
{
  \small	
  \textbf{\textit{Keywords---}} #1
}

\begin{abstract}
On the generalized tangent bundle of a smooth manifold, we study skew-symmetric endomorphisms satisfying an arbitrary polynomial equation with real constant coefficients. We investigate the compatibility of these structures with the de Rham operator and the Dorfman bracket. In particular, we isolate several conditions that when restricted to the motivating example of generalized almost complex structure are equivalent to the notion of integrability. 
\end{abstract}
\maketitle

\section{Introduction}
Let $M$ be a differentiable manifold. An {\it almost complex structure} on $M$ is an endomorphism $f$ of the tangent bundle $TM$ such that $f^2+I=0$, where $I$ denotes the identity endomorphism. It is easy to see that in order to support an almost complex structure, $M$ has to be even-dimensional. The odd-dimensional analogue is an {\it almost contact structure}, i.e.\ an endomorphism $f$ of $TM$ together with a vector $E$ and a 1-form $\eta$ such that $f^2+I=E\otimes \eta$. Yano \cite{Yano63} introduced the notion of {\it $f$-structure} i.e.\ an endomorphism of $TM$ such that $f^3+f=0$, of which almost complex structures and almost contact structures are both particular cases. In a different direction, {\it almost tangent structures} \cites{Eliopoulos65, Lehmann-Lejeune66} are endomorphisms $f$ of $TM$ such that $f^2=0$. Motivated by the quartic case, Goldberg and Yano \cite{GoldbergYano70} (see also \cite{GoldbergPetridis73}) looked at general {\it polynomial structures} i.e.\ endomorphisms $f$ of $TM$ satisfying $P(f)=0$ for some polynomial $P(x)$ with real constant coefficients and such that $I, f,\ldots,f^{\deg(P)-1}$ induce linearly independent endomorphisms on each fiber of $TM$.    

Hitchin \cite{Hitchin03} has shown that certain geometric structures arising from the study of the supersymmetric sigma-model can be understood as {\it generalized almost complex structures}, i.e.\ as endomorphisms $\varphi$ of the generalized tangent bundle $\T M= TM\oplus T^*M$ that are skew-symmetric with respect to the tautological inner product and such that $\varphi^2+I=0$. Any classical almost complex structure $f$ gives rise to a generalized almost complex structure by setting $\varphi=f\oplus (-f^{t})$. Just as in the classical case, generalized almost complex structures can only exist on even-dimensional manifolds. {\it Generalized $F$-structures} i.e.\ skew-symmetric endomorphisms $\varphi$ of the generalized tangent bundle such that $\varphi^3+\varphi=0$ were introduced by Vaisman \cite{Vaisman08} as a common extension of the notions of generalized almost complex structure and of classical $f$-structure.

The goal of the present paper is to fill a gap in the literature and study the analogue of classical polynomial structures in the context of generalized geometry i.e.\ skew-symmetric endomorphisms $\varphi$ of the generalized tangent bundle such that $P(\varphi)=0$ for some polynomial $P(x)$ with real constant coefficients. We call these structures {\it generalized polynomial structures}. In addition to generalized almost complex structures and generalized $F$-structures, the quadratic case has recently received some attention with progress being made for {\it generalized almost tangent structures} \cite{BlagaCrasmareanu14}, and more generally for {\it generalized metallic structures} \cite{BlagaNannicini20}. With the exception of the somewhat degenerate case of generalized almost tangent structures, a common feature of the particular examples of generalized polynomial structures studied in the literature up to this point is that the polynomial $P(x)$ has simple roots and thus $\T M\otimes \mathbb C$ splits into the direct sum of the corresponding eigenbundles. In general, this is no longer the case and $\varphi$ may very well be non-diagonalizable. 

Our first main observation, Theorem \ref{theorem:40} in Section \ref{sec:3}, is that given a generalized polynomial structure $\varphi$, the complexified generalized tangent bundle decomposes into the direct sum of the generalized eigenbundles $L_\lambda$ of $\varphi$. Moreover, each $L_\lambda$ is an isotropic eigenbundle and for each non-zero $\lambda$ the restriction of the tautological inner product to $E_\lambda=L_\lambda\oplus L_{-\lambda}$ is non-degenerate. In particular, every generalized polynomial structure admits a Jordan-Chevalley decomposition into its semisimple and nilpotent parts, both of which are also generalized polynomial structures. Another important consequence is that every generalized polynomial structure induces a canonical (up to overall shifts and involutions) multigrading on (complexified) differential forms. In the particular case of generalized almost complex structures, we recover the familiar decomposition of the complexified generalized tangent bundle into the $\pm\sqrt{-1}$-eigenbundles and the corresponding grading of complex differential forms \cite{Gualtieri11} (later extended to generalized F-structures in \cite{AldiGrandini17}). Finally, using results by Burgoyne and Cushman \cite{BurgoyneCushman77} in the case where $M$ is a point, we offer a detailed description of the possibly indecomposable blocks that can arise.

Generalized complex structures are characterized \cites{Gualtieri11, Cavalcanti06} among all generalized almost complex structures in three equivalent ways: 1) by the vanishing of their Courant-Nijenhuis torsion, 2) by the closure of their eigenbundles with respect to the Dorfman bracket, and 3) the decomposition $d=\partial + \overline \partial$, with $\partial$ (respectively $\overline \partial$) of degree $1$ (respectively $-1$) with respect to the induced grading on complex differential forms. In Theorem \ref{theorem:56} of Section \ref{sec:4}, our second main result, we extend this characterization to arbitrary generalized polynomial structures. Naively, based on the example of generalized complex structures and on the definition of integrability for classical polynomial structures \cites{GoldbergPetridis73,IshiharaYano65, Vanzura76}, one might expect that imposing the vanishing of the Courant-Nijenhuis torsion leads to an interesting class of generalized polynomial structures. Unfortunately, as we point out in Section \ref{sec:4} this condition is too strong as it forces the generalized polynomial structure to have at most two distinct eigenvalues. This is in sharp contrast with the classical case, where the vanishing of the Nijenhuis torsion imposes no limit on the number of distinct eigenvalues of a polynomial structure. As it turns out, this discrepancy is closely related to the fact that if the restriction of the tautological inner product to a proper subbundle $E\subseteq \T M\otimes \mathbb C$ is non-degenerate, then $E$ cannot be closed under the Dorfman bracket. Hence, the natural extension of condition 2) above is the requirement that if $\lambda\neq - \mu$ then the Dorfman bracket $\lbra L_\lambda, L_\mu \rbra$ is contained in $L_\lambda+L_\mu$. In particular no condition is imposed on the Dorfman bracket of two sections of $L_0$, the generalized eigenbundle corresponding to the eigenvalue $0$. Correspondingly, the natural extension of condition 1) to arbitrary generalized polynomial structures is the vanishing of what we call the {\it shifted Courant-Nijenhuis torsion} of the semisimple part. The appearance of the shifted Courant-Nijenhuis is not entirely unexpected because its vanishing precisely encodes the notion of {\it strong integrability} of generalized F-structures \cites{PoonWade11, AldiGrandini15}. The fact that only the semisimple part plays a role in Theorem \ref{theorem:56} is also not entirely unexpected since the generalized eigenbundles and the multigrading of differential forms that they induce only depend on the semisimple part.

Having obtained in Section \ref{sec:4} a satisfactory understanding of how the semisimple part of a generalized polynomial structure interacts with the de Rham operator, we would like to grasp the geometric meaning of the nilpotent part. We tackle this task in Section \ref{sec:5}. Our starting point is the observation \cites{Guttenberg07, Tomasiello08} that the integrability of a generalized almost complex structure $\varphi$ is equivalent to the vanishing of the derived bracket with respect to the operator $(\ad_{\widetilde \varphi}^2+I)(d)$, where $\widetilde \varphi$ is a lift of $\varphi$ to the space of $\mathbb R$-linear operators acting on differential forms and $\ad_{\widetilde \varphi}$ denotes the corresponding adjoint action (with respect to the natural graded commutator of operators acting on forms). As pointed out in \cite{AldiGrandini17} a similar result holds for a generalized F-structure $\varphi$, whose strong integrability is equivalent to the vanishing of the derived bracket with respect to the operator $(\ad_{\widetilde \varphi}^3+\ad_{\widetilde\varphi})(d)$. With this motivation in mind, in the present paper we introduce the {\it minimal torsion} $\mathcal M_\varphi$ of a generalized polynomial structure $\varphi$ with minimal polynomial $P(x)$ as the derived bracket for the operator $P(\ad_{\widetilde{\varphi}})(d)$. The terminology is justified by the fact that if $P(x)$ is quadratic (including the case of generalized almost complex structures discussed above) then $\mathcal M_\varphi=2\mathcal T_\varphi$, where $\mathcal T_\varphi$ is the usual Courant-Nijenhuis torsion of $\varphi$. Similarly, if $P(x)$ is cubic then $\mathcal M_\varphi = -3 \mathcal S_\varphi$, where $S_\varphi$ is the shifted Courant-Nijenhuis torsion introduced in Section \ref{sec:4}. Our first result in this direction, Theorem \ref{theorem:57}, explicitly relates the minimal torsion of a generalized polynomial structure to what we call {\it (shifted) higher Courant-Nijenhuis torsions} (i.e.\ the natural extension of higher Nijenhuis and Haantjes torsions \cites{Kosmann-Schwarzbach19, TempestaTondo18a, TempestaTondo21} to sections of the generalized tangent bundle). Our second result on minimal torsion, Theorem \ref{theorem:83}, characterizes the vanishing of the minimal torsion in terms of a suitable decomposition of the de Rham operator. 

We say that a generalized polynomial structure is {\it minimal} if its minimal torsion vanishes. Our perspective is that minimality provides a useful notion of compatibility with the structure of Courant algebroid on the generalized tangent bundle that involves both the semisimple and the nilpotent parts. A natural question at this point is how minimality relates to the compatibility conditions, which do not involve the nilpotent part, discussed in Section \ref{sec:4}. As it turns out, for a generic generalized polynomial structure, minimality implies that the semisimple part satisfies the equivalent conditions of Theorem \ref{theorem:56}. As a consequence, we have that, generically, the decomposition of the de Rham operator provided by Theorem \ref{theorem:83} coincides with the decomposition into multigraded components guaranteed by Theorem \ref{theorem:56}. The non-generic case is characterized by a sort of ``resonance'' condition in which an eigenvalue can be written as a non-trivial sum of three other eigenvalues. This phenomenon is rather unexpected and it would be interesting to understand it better. Here we limit ourselves to show, by means of an example, that this condition cannot be entirely removed. 

Throughout the paper, we place particular emphasis on the geometrically relevant \cite{GoldbergYano70} quartic case. In particular, we show that minimal generalized polynomial structures with minimal polynomial $(x^2+1)^2$ are nothing but commuting pairs consisting of a generalized complex structure and a weak generalized tangent structure. Furthermore, minimal generalized polynomial structures with minimal polynomial $x^2(x^2+1)$ are nothing but commuting pairs consisting of a strongly integrable generalized F-structure and a weak generalized tangent structure with vanishing bivariate Courant-Nijenhuis torsion. 

In Section \ref{sec:6} we further illustrate the main constructions of the paper with concrete examples of polynomial structures on low-dimensional Lie groups. We treat these cases by constructing a basis of the underlying Lie algebra that is adapted to the block decomposition of a given generalized polynomial structure. This allows us to explicitly compute the higher Courant-Nijenhuis torsions and their shifted counterparts. While our focus is on particular examples, our calculations are algorithmic in nature and could conceivably be carried out in a systematic fashion with the aid of suitable computer algebra systems. We hope to come back to this point in future work.

\vskip.3in

 {\bf Acknowledgements:} We would like to thank the anonymous referees for the many useful suggestions.

\section{Preliminaries}

The goal of this section is to make the paper reasonably self-contained by providing the necessary background on the geometry of the generalized tangent bundle of a smooth manifold. We find it convenient to systematically employ the language, developed in \cite{AldiGrandini17} (see also \cite{Buttin74} for the classical case), of $\mathbb R$-linear operators acting on the sheaf of differential forms. The well-known concept of derived bracket (\cites{Vinogradov90, Kosmann-Schwarzbach04}) is of crucial importance in the reminder of the paper.

\subsection{Operators acting on differential forms}

Let $M$ be a smooth manifold of dimension $n$. We denote by $\Omega_M$ the sheaf of differential forms on $M$ and by $\mathcal E_M$ the sheaf of algebras of $\mathbb R$-linear endomorphisms of $\Omega_M$. The usual grading $\Omega_M=\bigoplus_{i=0}^n \Omega_M^i$ induces a compatible grading on $\mathcal E_M$. Sections of $\mathcal E_M$ are endowed with the natural composition product, giving rise to the standard (graded) commutator $[\psi,\varphi]=\psi\varphi-(-1)^{kl}\varphi\psi$. Each section $\zeta \in \mathcal E_M^l$ acts on $\mathcal E_M$ as a left-derivation of degree $l$ via the adjoint action $\ad_\zeta(\varphi)=[\zeta,\varphi]$.

\begin{example}
By letting differential forms act by left multiplication $\eta \omega=\eta \wedge \omega$ (the wedge notation will be omitted), $\Omega_M$ can be canonically identified with a subsheaf of $\mathcal E_M$ . With this identification in mind, we observe that $[\Omega_M,\Omega_M]=0$.   
\end{example}

\begin{example}
The tangent sheaf $\mathcal X_M^1$ is canonically identified with the subsheaf of $\mathcal E_M$ consisting of those local sections $X$ such that $X(1)=0$ and $[X,\omega]=\iota_X \omega$, for any $\omega\in \Omega_M$. Taking linear combinations of compositions of vectors (now identified with sections of $\mathcal E_M$) yields the full sheaf $\mathcal X_M$ of all polyvector fields acting on $M$. It is easy to see that $\mathcal X_M$ and $\mathcal X_M\oplus \Omega_M$ are both closed under the commutator.
\end{example}

\begin{example}
The de Rham differential is a global section of $\mathcal E_M^1$. In our notation, $[d,d]=2d^2=0$, $[X,d]=\mathcal L_X$ for every $X\in TM$, and $[d,\omega]=d\omega$ for every $\omega\in \Omega_M$.  
\end{example}

\begin{mydef}
Let $D$ be an $\mathbb R$-linear endomorphism of $\mathcal E_M$, of degree $k$. The {\it derived bracket} associated to $D$ is the bilinear operation 
\begin{equation}
\lbra \psi, \varphi\rbra_D = (-1)^{kl+1}[D\psi, \varphi]\,,
\end{equation}
for all $\psi\in \mathcal E^l_M$ and $\varphi\in \mathcal E_M$.  In particular, if $D=\ad_\delta$, for some $\delta\in \mathcal E_M$ we use the shorthand notation $\lbra \psi,\varphi \rbra_\delta$ for $\lbra \psi,\varphi \rbra_{\ad_\delta}=[[\psi,\delta],\varphi]$. 
\end{mydef}

\begin{example}
If $\delta=d$, the resulting bracket is known as the {\it Dorfman bracket}, for which we will reserve the special notation $\lbra\,\, ,\,\, \rbra$. More generally, if $H$ is a closed 3-form on $M$ then the derived bracket corresponding to $\delta=d+H$ is known as a {\it twisted Dorfman bracket}. For simplicity, in the rest of the paper we mostly focus on the untwisted case. 
\end{example}

\subsection{The generalized tangent bundle}

The {\it generalized tangent bundle} of $M$ is the bundle $\mathbb T M$ associated to the sheaf $\mathcal X_M^1\oplus \Omega_M^1$. The {\it tautological inner product} is the symmetric, $\Omega_M^0$-bilinear map $\langle \,\, , \,\,\rangle$ defined by the formula
\begin{equation}
\langle \x,\y\rangle = \frac{1}{2} [\x,\y]
\end{equation}
for all $\x,\y \in \Omega_M$. 

\begin{rem}
The generalized tangent bundle is closed under any twisted Dorfman bracket. 
\end{rem}

\begin{lem}\label{lemma:7}
Let $\varphi,\delta$ be sections of $\mathcal E_M$, with $\delta$ of odd degree. Then
\begin{enumerate}[1)]
    \item $[\varphi,\mathbb T M]\subseteq \Omega_M^0$ if and only if $\varphi\in \Omega_M^0\oplus \mathbb T M$;
    \item $[\varphi,\mathbb T M]=0$ if and only if $\varphi\in \Omega_M^0$;
    \item $\lbra \T M,\T M\rbra_\delta =0$ if and only if $\delta\in \T M$. 
\end{enumerate}
\end{lem}

\begin{proof}The backward implications are all straightforward from the definitions.
Assume without loss of generality that $\varphi\neq 0$ and that is of degree $k$. If $[\varphi,\T M]\subseteq \Omega_M^0$, then 
\begin{equation}
[\varphi,[\alpha,X]]=[[\varphi,\alpha],X]+(-1)^k[\alpha,[\varphi,X]]=0
\end{equation}
for every $\alpha\in \Omega^1_M$, $X\in \mathcal X_M^1$. Since every function is locally expressible as the tautological inner product of a 1-form and a vector, we conclude that $\left[\varphi,\Omega_M^0\right]=0$. Since $\iota_X\varphi(1)\in \Omega_M^0$ for all $X\in \mathcal X_M^1$  we have $\varphi(1)\in \Omega^1_M\oplus \Omega^0_M$. Let $\varphi(1)\neq 0$: in this case $k\in\{0,1\}$ and by degree comparison, $\left[\varphi, \Omega_M^1\right]=0$, which in turn implies
that $\varphi(\omega)=\varphi(1)\omega$ for all $\omega\in \Omega_M$, i.e. $\varphi\in \Omega_M^k$ with $k\in\{0,1\}$. Now, let $\varphi(1)=0$:
for all $\alpha\in \Omega_M^1$ we have that the linear map $\alpha\mapsto\varphi(\alpha)\in\Omega_M^0$ is tensorial and hence of the form $\iota_X$ for some $X\in \mathcal X_M^1$. Also note that from $\varphi\neq 0$ and $\left[\varphi,\Omega_M^1\right]\subseteq \Omega_M^0$ follows $X\neq 0$, so that $k=-1$. Now, again by degree comparison $\left[\varphi,\mathcal X^1_M\right]=0$ whence $\varphi(\omega)=\iota_X\omega$ for all $\omega\in \Omega_M$. This concludes the proof of 1). The proofs of 2) and 3) follow immediately.  
\end{proof}

\begin{lem}\label{lem:8}
Let $\delta\in \mathcal E_M$ be $\Omega_M^0$-linear and odd. Then the restriction of the derived bracket associated to $\delta$ to $\T M$ is $\Omega_M^0$-bilinear and skew-symmetric.
\end{lem}

\begin{proof} To see that the derived bracket is skew-symmetric we calculate 
\begin{equation}
\lbra \x,\y\rbra_\delta-\lbra \y,\x\rbra_\delta= [\delta,[\x,\y]]=0
\end{equation}
for all $\x,\y\in \T M$. Similarly, for all $\x,\y\in \T M$ and for all $f\in \Omega_M^0$
\begin{equation}
\lbra f \x,\y\rbra_\delta -f\lbra \x,\y\rbra_\delta = [[f,\delta]\x,\y]=0 \,.
\end{equation}
\end{proof}

\begin{mydef}
A {\it skew-symmetric endomorphism} of $\T M$ is an $\Omega_M^0$-linear bundle endomorphism $\varphi$ of $\T M$ such that $\langle \varphi(\x),\y\rangle + \langle \x, \varphi(\y)\rangle =0$ for all $\x,\y\in \T M$.
\end{mydef}

\begin{prop}\label{prop:9}
Let $\varphi$ be a skew-symmetric endomorphism of $\T M$. Then there is a unique even element $\widetilde \varphi\in \mathcal E_M$ such that
\begin{enumerate}[1)]
\item $[\widetilde \varphi, \Omega_M^0]=0$;
\item $[\widetilde \varphi,\x]=\varphi(\x)$ for all $\x\in \T M$;
\item $\widetilde \varphi(1)\in \Omega^2_M$.
\end{enumerate}
\end{prop}

\begin{proof}
Define $\widetilde \varphi(1)=\epsilon$, where $\epsilon\in \Omega_M^2$ is uniquely defined by $\iota_X\epsilon=-\varphi(X)(1)$ for all $X\in \mathcal{X}_M^1$ (note that $\epsilon$ is alternating precisely because $\varphi$ is skew-symmetric).
Condition 1) suggests to define $\widetilde \varphi(f)=f\epsilon$ for all $f\in \Omega_M^0$. Moreover, condition 2) applied to all $\x\in \Omega_M^1$ inspires to extend recursively $\widetilde \varphi$ to all of $\Omega_M$:
\begin{equation}\label{eq:1}
\widetilde \varphi(\eta\gamma)=\eta\widetilde \varphi(\gamma)+\varphi(\eta)(\gamma)
\end{equation}
for all $\eta\in \Omega_M^1$ and $\gamma\in \Omega_M^k$. In order to show that $\widetilde \varphi$ so constructed is the required element it remains only to prove that $\widetilde \varphi$ satisfies condition 2) for all $\x\in \mathcal{X}_M^1$ as well. Let $X\in \mathcal X_M^1$ and let $\omega\in \Omega_M^k$. We need to prove that
\begin{equation}\label{eq:prop10}\widetilde\varphi(\iota_X\omega)-\iota_X\widetilde\varphi(\omega)=\varphi(X)(\omega)\end{equation}
by induction on $k$: if $k=0$ the identity (\ref{eq:prop10}) follows from condition 1) and the definition of $\epsilon$. Suppose now $k>0$ and that  (\ref{eq:prop10}) holds for $0\leq h<k$. Then, we can assume $\omega=\alpha\beta$, with $\alpha\in \Omega_M^1$ and $\beta\in\Omega_M^{k-1}$. Then
\begin{eqnarray} \widetilde\varphi(\iota_X\omega)-\iota_X\widetilde\varphi(\omega)&=& \widetilde\varphi(\iota_X(\alpha\beta))-\iota_X\widetilde\varphi(\alpha\beta)\nonumber\\
&=& \widetilde\varphi((\iota_X\alpha)\beta-\alpha(\iota_X\beta))-\iota_X\widetilde\varphi(\alpha\beta)\nonumber\\
&=& \widetilde\varphi((\iota_X\alpha)\beta)-\widetilde\varphi(\alpha(\iota_X\beta))-\iota_X\widetilde\varphi(\alpha\beta)\nonumber\\
&=& (\iota_X\alpha)\widetilde\varphi(\beta)-\alpha\widetilde\varphi(\iota_X\beta)-\varphi(\alpha)(\iota_X\beta)-\iota_X\left(\alpha\widetilde\varphi(\beta)+\varphi(\alpha)(\beta)\right)\nonumber\\
&=&\alpha\iota_X\widetilde\varphi(\beta)-\alpha\widetilde\varphi(\iota_X\beta)-\varphi(\alpha)(\iota_X\beta)-\iota_X\left(\varphi(\alpha)(\beta)\right)\nonumber\\
&=&- \alpha\varphi(X)(\beta)-[\varphi(\alpha),X](\beta)\nonumber\\
&=&- \alpha\varphi(X)(\beta)+[\alpha,\varphi(X)](\beta)\nonumber\\
&=&\varphi(X)(\omega),\nonumber
\end{eqnarray}
which concludes the inductive step and the proof of the existence of the required element. The uniqueness of $\widetilde \varphi$ is an immediate consequence of conditions 1), 2), 3), and Lemma \ref{lemma:7}. \end{proof}

\begin{rem}\label{rem:11}
Condition 3) in Proposition \ref{prop:9} is somewhat ad hoc, its sole purpose being to ensure uniqueness. Any two operators  satisfying both 1) and 2), which we may think of as different lifts of $\varphi$ to $\mathcal E_M$, differ by addition of a section of $\Omega_M^0$.
\end{rem}

\subsection{Quasi split structures and gradings}
A {\it quasi split structure} is a subbundle $E$ of $\T M\otimes \C$ that satisfies the following requirements:
\begin{enumerate}[i)]
\item $E$ is {\it non-degenerate}, i.e.\ the restriction of the tautological inner product to $E$ is non-degenerate;
\item $E$ is {\it split}, i.e.\ $E$ splits fiberwise as direct sum of two isotropic subspaces.
\end{enumerate}
\begin{rem} Every non-degenerate subbundle $E\subseteq \T M\otimes \mathbb C$ is a quasi split structure if and only if the rank of $E$ is equal to twice the {\it Witt index} (i.e.\ the maximum dimension of an isotropic subspace) of $E_p$, for each $p\in M$.
\end{rem}
\begin{rem}
Let $E$ be a subbundle of $\T M$. If the restriction of the tautological inner product to $E$ has signature $(k,k)$ for some $k\in \{1,\ldots, \dim(M)\}$, then $E\otimes \C$ is a quasi split structure.  Quasi split structures of this form are called {\it split structures} \cites{AldiGrandini16, AldiGrandini17}.
\end{rem}

\begin{rem}\label{rem:13}
A feature of the Dorfman bracket is that if $f\in \Omega_M^0$ and $\x,\y\in \T M$, then \begin{equation}\label{eq:2}
\lbra f\x,\y\rbra = f\lbra \x,\y\rbra - 2\langle d f,\y\rangle \x  + 2\langle \x,\y \rangle d f\,.
\end{equation}
The last term of \eqref{eq:2} is not present if the Courant bracket (i.e.\ the skew-symmetrization of the restriction of the Dorfman bracket to the generalized tangent bundle) is used instead. In particular, this shows that an isotropic subbundle $L\subseteq \T M$ is closed under the Courant bracket if and only if it is closed under the Dorfman bracket. However for non-isotropic subbundles, the last term of \eqref{eq:2} does not vanish in general. 
\end{rem}

\begin{lem}\label{lem:14}
$\T M\otimes \mathbb C$ is the only non-zero  quasi split structure that is closed under the Dorfman bracket.
\end{lem}

\begin{proof}
Let $E$ be a non-zero quasi split structure. By non-degeneracy, $E$ admits (in a fixed local neighborhood) at least two sections $\x, \y$ such that $\langle \x,\y\rangle\neq 0 $ and the projection of $\x$ onto $TM$ is non-zero. Since $f$ in \eqref{eq:2} can be chosen arbitrarily, we have $T^*M\subseteq E$ and thus $E=\T M$.  
\end{proof}

\begin{rem}
Let $E$ be a quasi split structure and let $E^\perp$ be its orthogonal complement in $\mathbb T M\otimes\C$ with respect to the tautological inner product. If $L$ is a maximal isotropic subbundle of $E$, then any maximal isotropic subspace of $\T M\otimes \C$ that contains $L$ is of the form $L\oplus L_{\perp}$, for some $L_{\perp}$ maximal isotropic in $E^{\bot}$. In particular, $E^{\bot}$ is also a quasi split structure: in fact, the Witt index of $E^{\perp}$ satisfies
\begin{equation}
{\rm Witt}(E^{\bot})=n-{\rm Witt}(E)=n-\frac{1}{2}{\rm rk}(E)=\frac{1}{2}{\rm rk}(E^{\bot})\,.
\end{equation}
\end{rem}

\begin{rem}
The formula $\x \x \omega = \langle \x, \x\rangle \omega$, valid for every $\x \in \T M \otimes \mathbb C$ and every $\omega\in \Omega_M\otimes \mathbb C$, shows the existence of a $\mathbb C$-linear map from the Clifford algebra $C\ell (\T M\otimes \mathbb C)$ to $\mathcal E_M\otimes \mathbb C$ known as the {\it standard spin representation} \cites{Gualtieri11, Meinrenken13}. The standard spin representation is irreducible, surjective, and faithful. It provides a fiberwise identification of $C\ell (\T M\otimes \mathbb C)$ with $\mathcal E_M\otimes \mathbb C$. Given a quasi split structure $E$, the standard spin representation restricts to two commuting actions of $C\ell(E)$ and $C\ell(E^\perp)$. 
\end{rem}

\begin{mydef}
Let $E\subseteq \T M\otimes \mathbb C$ be a quasi split structure and let $L$ be maximal isotropic in $E$. The {\it canonical bundle} of $L$ is the subsheaf $K_L\subseteq \Omega_M\otimes \mathbb C$ of all sections $\rho$ such that $l\rho=0$ for all $l\in L$.
\end{mydef}

\begin{rem}\label{rem:14}
Fiberwise, the canonical bundle is an example of a pure subspace in the sense of \cite{Batista14}. As proved in \cite{Batista14}, $K_L$ has rank $2^{\dim(M)-{\textrm{rank}}(L)}$, showing in particular that $K_L$ is indeed a vector bundle. Also proved in \cite{Batista14} is the identification $L={\rm Ann}(K_L)$.
\end{rem}

\begin{prop}\label{prop:16}
Let $E$ be a quasi split structure of rank $2k$. Each decomposition $E=L\oplus L'$ into maximal isotropic subspaces gives rise to a canonical grading
\begin{equation}\label{eq:grading}
\Omega_M\otimes \mathbb C=\bigoplus_{r=0}^k (\BigWedge^r L')\otimes K_L\,.
\end{equation}
\end{prop}

\begin{proof}
Pick a decomposition $E^\perp=L_\perp\oplus L_\perp'$ into maximal isotropic subspaces. By construction, $C\ell(E^\perp)$ acts on $K_L$. The $C\ell(E^\perp)$-submodule generated by the pure spinor line $K_{L\oplus L_\perp}$ is  (see e.g.\ Proposition 3.10 in \cite{Meinrenken13}) isomorphic to the spinor module $\BigWedge^{\bullet} L_\perp'$. Since the latter is of rank $2^{\dim(M)-k}$, we conclude from Remark \ref{rem:14} that $K_L=(\BigWedge^\bullet L_\perp')\otimes  K_{L\oplus L_\perp}$. Hence the grading  \eqref{eq:grading} is obtained from the canonical grading (see e.g.\ \cite{Gualtieri11}) of $\Omega_M\otimes \mathbb C$ induced by the maximal isotropic decomposition $\T M\otimes \mathbb C=(L\oplus L_\perp)\oplus (L'\oplus L_\perp')$ upon forgetting the information coming from the the choice of $L_\perp$ and $L_\perp'$.
\end{proof}

\begin{mydef}
We refer to the grading \eqref{eq:grading} as the {\it $(L,L')$-grading}. 
\end{mydef}

\begin{rem}
As shown in \cite{Gualtieri11}, $K_{L'\oplus L_\perp'}=\det(L'\oplus L_\perp')\otimes K_{L\oplus L_\perp}$. Acting on both sides by $C\ell(E^\perp)$, we obtain the identification $K_{L'}=\det(L')\otimes K_L$ and thus of $(\BigWedge^{k-r} L)\otimes K_{L'}$ with $(\BigWedge^r L) \otimes K_{L'}$. Hence, the $(L,L')$-grading is related to the $(L',L)$-grading by an overall involution. 
\end{rem}

\begin{cor}\label{cor:19}
Consider a decomposition  $\T M\otimes \mathbb C=\bigoplus_{i=1}^s E_i$ into quasi split structures $E_1,\ldots, E_s$ of respective rank $2k_1,\ldots,2k_s$. The additional datum of isotropic decompositions $E_i=L_i\oplus L'_i$ for all  $1\le i \le t\le s$ gives rise to a canonical multigrading
\begin{equation}\label{eq:multigrading}
\Omega_M\otimes \mathbb C = \bigoplus_{r_1=0}^{k_1}\cdots\bigoplus_{r_t=0}^{k_t}(\BigWedge^{r_1} L_1')\otimes \cdots \otimes (\BigWedge^{r_t} L_t')\otimes K_{L_1\oplus \cdots \oplus L_t} \,.   
\end{equation}
\end{cor}

\begin{proof}
This follows immediately from Proposition \ref{prop:16} applied to the $(L_1\oplus\cdots \oplus  L_t,L_1'\oplus\cdots \oplus L_t')$-grading. 
\end{proof}

\begin{lem}\label{lem:23}
Let $\delta\in \mathcal E_M\otimes \mathbb C$ be an operator such that $[\delta,\T M\otimes \mathbb C]\subseteq \T M\otimes \mathbb C$ and let $L,L'\subseteq \T M\otimes \mathbb C$ be isotropic subbundles such that $L\oplus L'$ is a quasi split structure of rank $2k$. If $\lbra L,L\rbra_\delta\subseteq L$ and $\lbra L',L'\rbra_\delta\subseteq L'$, then $\delta$ decomposes into components of degree $-1,0,1$ with respect to the $(L,L')$-grading.
\end{lem}

\begin{proof}
Our proof is modeled after the proof of Theorem 2.9 in \cite{Gualtieri11}. Let $F_r\subseteq \Omega_M\otimes \mathbb C$ consist of all forms annihilated by $\BigWedge^{r+1}L$ so that, in particular, $F_0=K_L$. The assumption $\lbra L,L\rbra_\delta\subseteq L$ implies $\lbra L,L\rbra_\delta F_0=0$ and, unraveling the definition of the derived bracket, $\delta(F_0)\subseteq F_1$. In a similar way, proceeding by induction on $r$, we prove that $\delta(F_r)\subseteq F_{r+1}$ for all non-negative integers $r$. Similarly, $\delta(F'_r)\subseteq F'_{r+1}$, where $\{F'_r\}_r$ is the filtration of $\Omega_M\otimes \mathbb C$ induced by $L'$. Hence, 
\begin{equation}
\delta((\BigWedge^r L')\otimes K_L)\subseteq \delta(F_r)\cap \delta(F'_{k-r})\subseteq F_{r+1}\cap F'_{k-r+1}
\end{equation}
which concludes the proof.
\end{proof}

\begin{lem}\label{lem:24}
Let $L,L'\subseteq \T M\otimes \mathbb C$ be isotropic subbundles such that $E=L\oplus L'$ is a quasi split structure. Assume $\delta \in \mathcal E_M$ is such that
\begin{enumerate}[1)]
    \item $\delta$ is of degree 1 with respect to a grading that commutes with the $(L,L')$-grading;
    \item $\lbra L,L'\rbra_\delta \subseteq E^\perp$;
    \item $\lbra L,L\oplus E^\perp\rbra_\delta\subseteq L$ and $\lbra L',L'\oplus E^\perp\rbra_\delta\subseteq L'$. 
\end{enumerate}
Then $\delta$ is of degree $0$ with respect to the $(L,L')$-grading.
\end{lem}

\begin{proof}
Let $\rho\in K_L$ be a (locally defined) pure spinor so that $L''={\rm Ann}(\rho)$ is maximal isotropic and contains $L$. Using 3) and expanding the derived bracket of $\delta$, we obtain 
\begin{equation}
    0=\lbra L'',L\rbra_\delta \rho = L''L\delta \rho\,.
\end{equation}
Hence there exists a local function $g$ such that $L\delta \rho = g\rho$. On the other hand, condition 1) and degree comparison force $g=0$ and thus $\delta(K_L)\subseteq K_L$. Let $\{F_r\}$ be the filtration of $\Omega_M\otimes \mathbb C$ defined by $L$ introduced in the proof of Lemma \ref{lem:23}. In particular, $LF_r=F_{r-1}$ and $L'F_r=F_{r+1}$. Let us assume $\delta(F_r)\subseteq F_{r}$ so that, as a consequence of 2), we have $\lbra L',L\rbra_\delta F_r\subseteq F_r$. On the other hand, expanding the definition of the derived bracket and using induction on $r$ we have 
\begin{equation}
\lbra L,L'\rbra_\delta F_r\subseteq L\delta F_{r+1} + F_r\,.
\end{equation}
Therefore, $\delta(F_{r+1})\subseteq F_{r+1}$ which, by induction on $r$, proves that $\delta(F_r)\subseteq F_r$ for all $r$. Using a similar argument for $L'$, we conclude that $\delta$ preserves the $(L,L')$-grading. 
\end{proof}

\section{Generalized Polynomial Structures}\label{sec:3}

In this section we introduce the notion of generalized polynomial structure, modeled after the classical polynomial structures of Goldberg and Yano \cite{GoldbergYano70}. We show that the generalized tangent bundle decomposes into the direct sum of generalized eigenbundles.  

\subsection{Definition and basic properties} 

\begin{mydef}
Let $N$ be a positive integer. A skew-symmetric endomorphism $\varphi$ of $\T M$ is called a \emph{generalized polynomial structure of degree} $N$  if there is a monic polynomial $P(x)\in \R[x]$ of degree $N$ such that, for all $p\in M$, $P(x)$ is the minimal polynomial of the endomorphism $\varphi_p:\T_p M\rightarrow \T_p M$. We will refer to the polynomial $P(x)$ as the \emph{minimal polynomial} of $\varphi$. 
\end{mydef}

\begin{example}
Let $f:TM\to TM$ be a polynomial structure with minimal polynomial $p(x)$. If $f^T$ denotes the transpose action on $T^*M$, then $\varphi=f\oplus (-f^T)$ is by construction a skew-symmetric endomorphism of $\T M$. Moreover, for every $X\in TM$ and $\alpha\in T^*M$, $0=\langle p(f)(X),\alpha\rangle = \langle X,p(f^T)\alpha\rangle$ implies that $p(\varphi)p(-\varphi)$ vanishes on $\T M$. Hence $\varphi$ is a generalized polynomial structure whose minimal polynomial $P(x)$ divides $p(x)p(-x)$. More precisely, if $q(x)$ denotes the largest (monic) even divisor of $p(x)$, then
\begin{equation*}
    P(x)=(-1)^{\deg(p)-\deg(q)}\frac{p(x)p(-x)}{q(x)}\,.
\end{equation*}
\end{example}

\begin{lem}\label{lem:GPS} 
Let $P(x)\in \R[x]$ be a monic polynomial of degree $N$, and let $\varphi$ be a skew-symmetric endomorphism of $\T M$. Then $\varphi$ is a generalized polynomial structure with minimal polynomial $P(x)$ if and only if the following conditions are satisfied:
\begin{enumerate}[1)]
\item $P(\varphi)(\x)=0$ for all $\x\in \T M$;
\item for all $p\in M$, the endomorphisms $I_p,\varphi_p,\varphi^2_p,\dots\varphi^{N-1}_p$ of $\T_p M$ are linearly independent.
\end{enumerate}
\end{lem} 

\begin{proof}
Suppose $\varphi$ is a generalized polynomial structure with minimal polynomial $P(x)$. Let us fix a point $p\in M$ and let $\x\in \T_p M$. Since $P(x)$ is the minimal polynomial of $\varphi_p$, then $P(\varphi_p)(\x)=0$. Moreover, suppose there exist real numbers $a_0,a_1,\ldots,a_{N-1}$ such that
\begin{equation}
a_0I_p+a_1\varphi_p+a_2\varphi_p^2+\dots+a_{N-1}\varphi_p^{N-1}=0.
\end{equation}
Then the polynomial $Q(x)=a_0+a_1x+a_2x^2+\dots+a_{N-1}x^{N-1}$ satisfies $Q(\varphi_p)=0$. Since ${\rm deg}\ Q(x) < {\rm deg}\ P(x)$, this forces $Q(x)=0$ i.e.\ $a_i=0$ for all $i$. Vice versa, suppose that conditions 1) and 2) are satisfied. It follows from 1) that $P(x)$ is multiple of $Q_p(x)$, the (monic) minimal polynomial of $\varphi_p$. On the other hand, 2) implies ${\rm deg}\ Q_p(x)\geq {\rm deg}\ P(x)$ and thus $Q_p(x)=P(x)$.
\end{proof}

\begin{rem}
Lemma \ref{lem:GPS} shows that generalized polynomial structures naturally extend the notion of classical polynomial structure as introduced by Goldberg and Yano \cite{GoldbergYano70} (and further studied in \cite{GoldbergPetridis73}) to the generalized tangent bundle.
\end{rem}

\begin{example} 
A generalized polynomial structure with minimal polynomial $P(x)=x^2+1$ is nothing but a {\it generalized almost complex structure} in the sense of Hitchin and Gualtieri \cites{Hitchin03,Gualtieri11}. Similarly, generalized polynomial structures with minimal polynomial $P(x)=x^3+x$ are precisely (non-zero) {\it generalized $F$-structures} in the sense of Vaisman \cite{Vaisman08} that are not generalized almost complex structures.
\end{example}

\begin{example} 
Generalized polynomial structures with minimal polynomial of the form $P(x)=x^2-p x-q$ are also known as skew-symmetric {\it generalized metallic structures} \cite{BlagaNannicini20}. (Throughout the present paper we will assume generalized metallic structures to be additionally skew-symmetric with respect to the tautological inner product, in contrast to \cite{BlagaNannicini20}.) In particular, {\it generalized almost tangent structures} \cite{BlagaCrasmareanu14} are generalized metallic structures such that $p=q=0$. 
\end{example}

\begin{rem}\label{rem:29}
Let $f$ be a (classical) metallic structure i.e.\ a polynomial structure with minimal polynomial $p(x)=x^2-ax-b$ and let $\varphi=f\oplus(-f^T)$. If $a=0$, then $\varphi$ is a generalized metallic structure with minimal polynomial $p(x)$. If $a\neq 0$, then $\varphi$ is a generalized polynomial structure of degree 4 whose with minimal polynomial is $p(x)p(-x)$.  
\end{rem}

\begin{rem}
Given a skew-symmetric endomorphism $\varphi:\T M\rightarrow \T M$, consider the ideal
\begin{equation}
{\mathcal I}(\varphi)=\{P(x)\in\R[x]: P(\varphi)(\x)=0\ \mbox{ for all }\x\in \T M\}\,.
\end{equation}
If $\varphi$ is a generalized polynomial structure, then its minimal polynomial $P(x)$ belongs to $\mathcal I(\varphi)$. In particular, ${\mathcal I}(\varphi)$ is non-trivial. Among all the elements of $\mathcal I(\varphi)$, $P(x)$ is the unique monic polynomial of smallest degree. This motivates the following. 
\end{rem}

\begin{mydef} \label{def: WGPS}
A skew-symmetric endomorphism $\varphi$ of $\T M$ is called a \emph{weak generalized polynomial structure} if ${\mathcal I}(\varphi)\neq 0$. The monic polynomial $P(x)$ of smallest degree in ${\mathcal I}(\varphi)$ is called the {\it minimal polynomial of $\varphi$}, while ${\rm deg}\ P(x)$ is called the {\it degree} of $\varphi$. 
\end{mydef}

\begin{lem} 
Let $M$ be a smooth connected manifold. A skew-symmetric endomorphism $\varphi$ of $\T M$ is a weak generalized polynomial structure if and only if its spectrum (the set of the eigenvalues, regardless of the multiplicities) is constant.
\end{lem}

\begin{proof} 
If $\varphi$ has a constant spectrum, then its characteristic polynomial $Q(x)={\rm det}(\varphi-xI)$ has constant coefficients and positive degree. Moreover, by the Cayley-Hamilton Theorem, $Q(x)\in {\mathcal I}(\varphi)$. Vice versa, if $Q(x)$ is any polynomial of positive degree in ${\mathcal I}(\varphi)$, then $Q(\lambda)=0$ for every eigenvalue $\lambda$ of $\varphi$. Since by definition all polynomials in $\mathcal I(\varphi)$ have constant coefficients, then the spectrum is also constant.
\end{proof}

\begin{lem}\label{lem:36}  
A skew-symmetric endomorphism $\varphi$ is a weak generalized polynomial structure if and only if there is a monic polynomial $P(x)\in \R[x]$ of positive degree $N$ such that 
\begin{enumerate}[1)]
    \item $P(\varphi)(\x)=0$ for all $\x\in\T M$;
    \item the endomorphisms $I,\varphi,\varphi^2,\dots,\varphi^{N-1}$ are $\R$-linearly independent.
\end{enumerate}
\end{lem}

\begin{proof} 
The proof is a straightforward adaptation of the proof of Lemma \ref{lem:GPS}.
\end{proof}

\begin{rem}
Lemma \ref{lem:36} shows that weak generalized polynomial structures of constant rank extend the notion of polynomial structure described by Goldberg and Yano \cite{GoldbergYano70} to the generalized tangent bundle.
\end{rem}

\begin{rem}
It follows from Remark \ref{rem:29} that every generalized polynomial structure is also a weak generalized polynomial structure (and the two definitions of the minimal polynomial coincide). The converse is not necessarily true as the monic polynomial of smallest degree in $\mathcal I(\varphi)$ is a priori only a multiple of the minimal polynomials of the fiberwise endomorphisms $\varphi_p$.
\end{rem}

\begin{example}
Given a generalized polynomial structure $\varphi$ with minimal polynomial $P(x)=x^2$ and a function $g\in \Omega_M^0$, consider $\varphi'=g\varphi$. Then, the minimal polynomial $Q_p(x)$ of $\varphi'_p$ satisfies
\begin{equation}
    Q_p(x)=\left\{\begin{array}{lc} x^2&\mbox{ if } g(p)\neq 0,\\
x&\mbox{ if } g(p)=0.\end{array}\right.
\end{equation}
In particular, suppose $g$ is non-zero and non-invertible. Then $\varphi'$ is not a generalized polynomial structure, but still a weak generalized polynomial structure. Moreover, $P(x)=x^2$ is still the minimal polynomial of $\varphi'(x)$ in the sense of Definition \ref{def: WGPS}.
\end{example}

\begin{lem}
Let $\varphi$ be a weak generalized polynomial structure such that all eigenvalues of $\varphi$ have constant multiplicity $1$ (i.e.\ $\varphi$ is semisimple). Then $\varphi$ is a generalized polynomial structure.
\end{lem}

\begin{proof}
Let $P(x)$ be the polynomial in ${\mathcal I}(\varphi)$ of smallest degree, and let $p\in M$. Since $P(\varphi_p)=0$ and $P(x)$ has only simple roots, then $P(x)$ is also the minimal polynomial of $\varphi_p$ for each $p\in M$.
\end{proof}

\begin{rem}\label{rem:34}
Let $\varphi$ be a weak generalized polynomial structure. Since $\varphi$ is in particular a skew-symmetric endomorphism of $\T M$, its spectrum must be symmetric with respect to $0\in \mathbb C$. As a consequence, the minimal polynomial $P(x)$ of $\varphi$ is either even or odd, i.e. of the form $P(x)=x^k Q(x^2)$, for some non negative integer $k$ and some $Q(x)\in \R[x]$. For example, imposing that a generalized metallic structure with minimal polynomial $x^2-px-q$ is skew-symmetric implies the condition $p=0$.
\end{rem}

\begin{example}
Let $f:T M\to T M$ be a polynomial structure and let $\lambda$ be an eigenvalue of the associated generalized polynomial structure $\varphi=f\oplus (-f^t)$. Then $L_\lambda=(L_\lambda\cap TM)\oplus (L_\lambda \cap T^*M)$. The first summand is the generalized eigenbundle with eigenvalue $\lambda$ for the action of $f$ of $TM$. The second summand is the generalized eigenspace with eigenvalue $-\lambda$ for the transpose action of $f^t$ on $T^*M$. 
\end{example}

\subsection{Generalized eigenbundles}
Let $\varphi$ be a weak generalized polynomial structure and let $\lambda$ be an eigenvalue of $\varphi$. We have a sequence of subsheaves of $\T M\otimes\C$:
\begin{equation}\label{filt1}
{\rm Ker}(\varphi-\lambda I)\subseteq {\rm Ker}((\varphi-\lambda I)^2)\subseteq {\rm Ker}((\varphi-\lambda I)^3)\subseteq \cdots\subset {\rm Ker}((\varphi-\lambda I)^n)\subseteq\cdots. 
\end{equation}
The inclusions in (\ref{filt1}) are all proper, until the sequence stabilizes. The smallest integer $m(\lambda)$ at which (\ref{filt1}) stabilizes coincides with the {\it multiplicity} of $\lambda$ in the minimal polynomial of $\varphi$. We denote 
\begin{equation}
    L_{\lambda}:={\rm Ker}((\varphi-\lambda I)^{m(\lambda)})
\end{equation}
the corresponding subsheaf. We also denote by $\Sigma(\varphi)$ the spectrum of $\varphi$ i.e.\ the set of all eigenvalues of $\varphi$. We reserve the notation $\Sigma_+(\varphi)$ for any fixed subset $\Sigma_+(\varphi)\subset \Sigma(\varphi)$ such that $\Sigma_+(\varphi)$ and $-\Sigma_+(\varphi)$ partition $\Sigma(\varphi)\setminus\{0\}$.

\begin{rem} 
If $\varphi$ is a generalized polynomial structure, then $L_{\lambda}$ is a vector subbundle of $\T M\otimes \C$, whose complex rank equals the multiplicity of $\lambda$ in the minimal polynomial. In this case we refer to $L_\lambda$ as the {\it generalized eigenbundle of $\varphi$ with eigenvalue $\lambda$}. In light of Remark \ref{rem:34}, $L_{\lambda}$ and $L_{-\lambda}$ have equal rank.
\end{rem}

\begin{rem}\label{rem:im}
Similarly to \eqref{filt1}, we have the filtration
\begin{equation}\label{filt2}
{\rm Im}(\varphi-\lambda I)\supseteq {\rm Im}((\varphi-\lambda I)^2)\supseteq {\rm Im}((\varphi-\lambda I)^3)\supseteq \cdots\supseteq {\rm Im}((\varphi-\lambda I)^n)\supseteq\cdots. 
\end{equation}
which stabilizes at $n=m(\lambda)$. If $\varphi$ is a generalized polynomial structure, then the sheaf
\begin{equation}
R_{\lambda}:={\rm Im}((\varphi-\lambda I)^{m(\lambda)})=\bigoplus_{\mu\in \Sigma(\varphi)\setminus\{\lambda\}}L_{\mu}
\end{equation}
is a vector bundle. More generally, given any subset
$S$ of $\Sigma(\varphi)$, we have
\begin{equation}
    {\rm Im}\left(\prod_{\lambda\in S}(\varphi-\lambda I)^{m(\lambda)}\right)=\bigoplus_{\mu\in \Sigma(\varphi)\setminus S}L_{\mu}
\end{equation}
and
\begin{equation}
\qquad{\rm Ker}\left(\prod_{\lambda\in S}(\varphi-\lambda I)^{m(\lambda)}\right)=\bigoplus_{\mu\in \Sigma(\varphi)\setminus S}R_{\mu}\,.
\end{equation}
\end{rem}

\begin{theorem} \label{theorem:40}
Let $\varphi$ be a generalized polynomial structure. The following properties are satisfied:
\begin{enumerate}[i)]
\item For all $\lambda\in\Sigma(\varphi)\setminus\{0\}$, the bundles $L_{\lambda}$ are isotropic;
\item for all $\lambda\in\Sigma(\varphi)\setminus\{0\}$, the bundles $$E_{\lambda}:=L_{\lambda}\oplus L_{-\lambda}$$ are quasi split structures;
\item the bundle $E_0:=L_0$ is a quasi split structure;
\item the bundle $\T M\otimes{\C}$ splits in the orthogonal sum
\begin{equation}\label{eq:splitting}\T M\otimes{\C}=E_0\oplus\bigoplus_{\lambda\in \Sigma_+(\varphi)}E_{\lambda}\end{equation}
\end{enumerate}
\end{theorem}

\begin{proof} It suffices to show that if $\lambda\in \Sigma(\varphi) $, then 
\begin{equation}
(L_{\lambda})^{\bot}=\bigoplus_{\mu\in \Sigma(\varphi)\setminus\{-\lambda\}} L_{\mu}\,.
\end{equation}
To see this, let $P(x)$ be the minimal polynomial of $\varphi$, and let
$F_{\lambda}(x)=P(x)/(x-\lambda)^{m(\lambda)}.$ From Remark \ref{rem:im} it follows that $L_{\lambda}$ is the image of $F_{\lambda}(\varphi)$.
Also note that $(F_{\lambda}(\varphi))^*=(-1)^{N-m(\lambda)}F_{-\lambda}(\varphi)$, where $N={\rm deg}\ P(x)$.
This implies that $\x\in (L_{\lambda})^{\bot}$ if and only if $\x\in {\rm Ker}( F_{-\lambda}(\varphi))$. Again from Remark \ref{rem:im}, the last condition is equivalent to 
$\x\in R_{-\lambda}$.
\end{proof}

\begin{cor}\label{cor:42}
Let $\varphi$ be a generalized polynomial structure that is neither invertible nor nilpotent, then $L_0$ is not closed under the Dorfman bracket.  
\end{cor}

\begin{proof}
Since $\varphi$ is not invertible, Theorem \ref{theorem:40} implies that $L_0$ is a non-zero quasi split structure. On the other hand $L_0$ is not nilpotent and thus a proper subbundle of $\T M$. The result then follows from Lemma \ref{lem:14}.
\end{proof}

\begin{rem}
The splitting \eqref{eq:splitting} does not determine uniquely $\varphi$. However, there is a unique polynomial structure $\varphi_s$ such that $L_{\lambda}$ is the $\lambda$-eigenbundle of $\varphi_s$. Setting $\varphi_n:=\varphi-\varphi_s$, the decomposition $\varphi=\varphi_s+\varphi_n$ coincides fiberwise with the Jordan-Chevalley decomposition of $\varphi$, where $\varphi_s$ is the {\it semisimple part of $\varphi$} and $\varphi_n$ is the {\it nilpotent part of $\varphi$}.
\end{rem}

\begin{rem}
The decomposition \eqref{eq:splitting} can be computed explicitly as follows. Let $a_{\lambda,i}$ be the coefficients of the partial fraction decomposition
\begin{equation}
    \frac{1}{P(x)}=\sum_{\lambda\in \Sigma(\varphi)}\sum_{i=1}^{m(\lambda)}\frac{a_{\lambda,i}}{(x-\lambda)^i}\,.
\end{equation}
For each $\lambda\in \Sigma(\varphi)$ and for each $i\in\{1,\ldots,m(\lambda)\}$, let $Q_{\lambda,i}$ be the polynomial defined by the relation $(x-\lambda)^i Q_{\lambda,i}(x)=P(x)$. Then the projector $\mathcal P_\lambda :\T M\otimes \mathbb C\to L_\lambda$ satisfies
\begin{equation}\label{eq:9}
    \mathcal P_\lambda(\x)= \sum_{i=1}^{m(\lambda)}a_{\lambda,i}Q_{\lambda,i}(\varphi)(\x)
\end{equation}
for all $\x\in \T M\otimes \mathbb C$. As a consequence of \eqref{eq:9}, the semisimple part of $\varphi$ can be written as
\begin{equation}\label{eq:44}
    \varphi_s = \sum_{\lambda \in \Sigma(\varphi)}\lambda \mathcal P_\lambda = \sum_{\lambda\in \Sigma(\varphi)} \sum_{i=1}^{m(\lambda)}\lambda a_{\lambda,i}Q_{\lambda,i}(\varphi)\,.
\end{equation}
In particular, we recover the known fact \cite{Humphreys78} that  the semisimple and the nilpotent parts of $\varphi$ are both polynomials in $\varphi$.
\end{rem}

\begin{example}
Let $c$ be a non-zero real number and let $\varphi$ be a generalized polynomial structure with minimal polynomial $P(x)=(x^2+c)^2$. The projectors on $L_{\pm \sqrt{-c}}$ are
\begin{align*}
\mathcal P_{\pm\sqrt{-c}}&=\pm\frac{1}{4c\sqrt{-c}}(\varphi\mp \sqrt{-c})(\varphi\pm \sqrt{-c})^2-\frac{1}{4c}(\varphi\pm \sqrt{-c})^2\\
&= \mp\frac{1}{4c\sqrt{-c}}(\varphi^3+3c\varphi\mp 2c\sqrt{-c})
\end{align*}
from which we calculate the semisimple part of $\varphi$ to be $\varphi_s=\frac{1}{2c}\varphi^3+\frac{3}{2}\varphi$. Hence the nilpotent part of $\varphi$ is $\varphi_n=-\frac{1}{2c}\varphi^3-\frac{1}{2}\varphi$.
\end{example}

\subsection{Block decomposition}\label{decompos}
The content of this subsection is adapted from \cite{BurgoyneCushman77}. Let $\varphi$ be a polynomial structure, and let $p\in M$. A \emph{real block}  of $\varphi$ at $p$ is a non-degenerate subspace $V\subseteq \T_p M$ such that $\varphi(V)\subseteq V$. Analogously, a \emph{complex block} of $\varphi$ at $p$ is a non-degenerate subspace $V\subseteq \T_p M\otimes\C$ such that $\varphi(V)\subseteq V$.
Moreover, a real block (respectively, a complex block) $V$ is called \emph{indecomposable} if there is no real block (respectively, complex block)  $V'$ such that $V'\subseteq V$ and $V'\neq V, V'\neq \{0\}$. Note that since $\varphi$ is skew, $\T_p M$ splits into an orthogonal sum of real indecomposable blocks, and similarly $\T_p M\otimes \C$ splits into an orthogonal sum of complex indecomposable blocks. Let ${\rm deg}(V)$ denote the degree of nilpotency of $\varphi_n$ on a block $V$, that is the smallest positive integer $k$ such that restriction of $\varphi_n^k$ to $V$ vanishes. 
The general structure of indecomposable blocks is described by the following
\begin{lem}\label{lem:block1}
For every indecomposable block $V$ of degree $k$, there exist a subspace $W\subset V$ satisfying the following properties:
\begin{enumerate}[(i)]
\item $V=W\oplus K$, where $K={\rm Ker}\left(\varphi_{n|V}^{k-1}\right)$;
\item $\varphi_s(W)\subset W$;
\item the bilinear form $\tau_{k-1}(\w,\w'):=\langle \w, \varphi_n^{k-1}(\w')\rangle$ (symmetric for $k$ odd and skewsymmetric for $k$ even) is non-degenerate on $W$;
\item $V=W\oplus \varphi_n(W)\oplus \varphi^2_n(W)\oplus\dots\oplus  \varphi^{k-1}_n(W)$;
\item if $V$ is a complex block and $\overline{V}=V$, then $\overline{W}=W$;
\item for any non negative integer $j<k-1$, the bilinear forms $\tau_{j}(\w,\w'):=\langle \w, \varphi_n^{j}(\w')\rangle$ vanish identically on $W$;
\item there are no proper subspaces $W'\subset W$, $W'\neq 0$ that satisfy properties (i)-(iii).
\end{enumerate}
\end{lem}
\begin{proof}
The existence of a subspace $W$ satisfying (i) and (ii) follows from $\varphi_s(K)\subset K$ and the fact that $\varphi_s$ is semisimple.
Now, we prove that $W$ satisfies (iii): if $\w\in W$ is such that $\tau_{k-1}(\w',\w)=0$ for all $\w'\in W$, then  $\langle \v, \varphi_n^{k-1}(\w)\rangle=0$ for all $\v\in V$, whence $\w\in W\cap K=0$. Moreover, note that $W$ generates a splitting
$$V':=W\oplus \varphi_n(W)\oplus\dots \oplus\varphi_n^{k-1}(W)\subseteq V.$$
Since $\tau_{k-1}$ is non-degenerate we deduce that $V'$ is a block, so that by indecomposability $V'=V$ and (iv) is proved. A similar reasoning shows (vii).
Next, we prove that we can modify  $W$ so that  (v) is satisfied as well: let $\overline{V}=V$, let $\sigma: V\rightarrow V$ be the complex conjugation map and let $\pi: V\rightarrow K$ the projection onto $K$ with respect to the decomposition $V=W\oplus K$. Now, we can replace $W$ with the graph  of the map 
$$f:W\rightarrow K, \qquad f=-\frac{1}{2}\sigma \pi\sigma_{|W}$$
which satisfies (i)-(v). It remains to prove that we can further modify $W$ so that (vi) is also satisfied. Let 
$\widehat{\tau}_j: W\rightarrow W^*$ such that $\widehat{\tau}_j(\w)(\w')=\tau_j(\w,\w')$. Suppose $j<k-1$ is such that $\widehat{\tau}_{h}=0$ for $j<h<k-1$.
Then, the graph $W''$ of the map
$$f_j:W\rightarrow K \qquad f_j=\frac{(-1)^{k-j}}{2}N^{k-1-j}\widehat{\tau}^{-1}_{h-1}\widehat{\tau}_j$$
is still a complement of $K$, satisfies properties (i)-(v) and additionally  $\tau_h(W'',W'')=0$ for  $j\leq h<k-1$. Replacing $W$ with $W''$ and repeating this construction, after a finite number of iterations we obtain the requested property. \end{proof}
\begin{mydef}
We will refer to a subspace $W$ satisfying the properties of the previous lemma as a \emph{semisimple component} of the block $V$.
\end{mydef}
\begin{lem}\label{lem:block2}
For any indecomposable complex block $V$ with ${\rm deg}(V)=k$ and semisimple component $W$, one of the following conditions is satisfied:
\begin{enumerate}[i)]
    \item $W\subset E_0$, $k$ is odd and ${\rm dim}(W)=1$. In particular, $W={\rm Span}_{\C}(\w)$ where $\w\in E_0$ and $\langle \w, \varphi_n^{k-1}(\w)\rangle=1$;
    \item $W\subset E_{\lambda}$ for some $\lambda\in\Sigma(\varphi)$ and ${\rm dim}(W)=2$. In particular, $W={\rm Span}_{\C}(\w_+,\w_-)$ where $\w_{\pm}\in L_{\pm\lambda}$ and $\langle \w_+, \varphi_n^{k-1}(\w_-)\rangle=\frac{1}{2}$. 
\end{enumerate}
\end{lem}
\begin{proof}
From the indecomposability $V\subset E_{\lambda}$ for some $\lambda$. If $\lambda\neq 0$, from Lemma \ref{lem:block1} follows that   $W=(W\cap L_{\lambda})\oplus(W\cap L_{-\lambda})$ where the summands are isotropic with respect to the bilinear form $\tau_{k-1}$, and that ${\rm dim}(W)=2$. If $\lambda=0$ and $k$ is even, $\tau_{k-1}$ is a indecomposable, non-degenerate skewsymmetric bilinear form, so that ${\rm dim}(W)=2$. If $\lambda=0$ and $k$ is odd, $\tau_{k-1}$ is a indecomposable, non-degenerate symmetric bilinear form, so that ${\rm dim}(W)=1$.
\end{proof}
\begin{rem} 
We denote by $\Delta_k(0)$ and  $\Delta_k(\lambda,-\lambda)$ the families of indecomposable blocks (also called \emph{ indecomposable complex types}) satisfying (i) and (ii) respectively in Lemma \ref{lem:block2}.  Note that for $k$ odd, the family  $\Delta_k(0,0)$ does not contain any indecomposable blocks. However, for $k$ odd we can redefine $\Delta_k(0,0):=\Delta_k(0)\oplus\Delta_k(0)$, that is the family of blocks that are orthogonal sum of two blocks of type $\Delta_k(0)$.
\end{rem}

\begin{rem}\label{ex: blocks} 
If ${\rm dim}\ M=3$, then at any $p\in M$ the complex block decomposition of $\varphi_p$ is of one of the following forms: 
\begin{enumerate}
    \item $\Delta_5(0)\oplus\Delta_1(0)$;
    \item $\Delta_2(\lambda_1,-\lambda_1)\oplus\Delta_1(\lambda_2,-\lambda_2)$;
    \item $\Delta_3(\lambda,-\lambda)$;
    \item $\Delta_3(0)\oplus\Delta_1(\lambda,-\lambda)\oplus\Delta_1(0)$;
  \item $\Delta_1(\lambda_1,-\lambda_1)\oplus \Delta_1(\lambda_2,-\lambda_2)\oplus \Delta_1(\lambda_3,-\lambda_3)$.
\end{enumerate}
Each decomposition corresponds to a minimal polynomial which may depend on the eigenvalues involved. Note that there are constraints on the eigenvalues: in the decompositions (2), (3) and (4) the eigenvalues must be either real or imaginary, while in (5) if $\lambda_i$ is neither real or imaginary, then there is a $\lambda_j$ with $j\neq i$ such that $\lambda_j=\pm\overline{\lambda_i}$ while the remaining eigenvalues $\pm \lambda_k$ are either real or imaginary.
\end{rem}
\noindent We would like to describe the indecomposable real blocks  as well. Let $V$ be a indecomposable complex block of degree $k$ and semisimple component $W$. Moreover, let $\tau=\tau_{k-1}$ be the non-degenerate bilinear form as before.\\ \\ Let first $V=\overline{V}$, so that $W=\overline{W}$ and $V$ is  of type either $\Delta_k(0)$ or $\Delta_k(\lambda,-\lambda)$ for $\lambda\in \R\cup\sqrt{-1}{\R}$. Then, $W=W_0\oplus\sqrt{-1}W_0$ and $V=V_0\oplus\sqrt{-1}V_0$, where $$V_0=W_0\oplus \varphi_{n}(W_0)\oplus \varphi^2_{n}(W_0)\oplus\dots\oplus \varphi^{k-1}_{n}(W_0)$$
and $V_0$ is a real indecomposable block, with semisimple part $W_0$. We denote by ${\rm sign}(V_0):=(n_+,n_-)$ the signature of the tautological inner product on $V_0$.

\begin{lem}\label{lem:block3} Let $V$ be an indecomposable complex block with semisimple part $W$ such that $\overline{V}=V$ and $\overline{W}=W$. One of the following conditions is satisfied:
\begin{enumerate}[(i)]
\item $V\in \Delta_{2h+1}(0)$: $W_0={\rm Span}_{\R}(\w)$ where $\varphi_s(\w)=0$ and $\langle \w, \varphi_n^{2h}(\w)\rangle=\epsilon=\pm 1$.
Moreover, when $\epsilon=+1$ we have ${\rm sign}(V_0)=(h+1,h)$ if $h$ is even, while  ${\rm sign}(V_0)=(h,h+1)$ if $h$ is odd.
When $\epsilon=-1$ we have ${\rm sign}(V_0)=(h,h+1)$ if $h$ is even, while  ${\rm sign}(V_0)=(h+1,h)$ if $h$ is odd.
\item $V\in \Delta_{k}(\lambda_0,-\lambda_0)$, with $\lambda_0\in \R$: $W_0={\rm Span}_{\R}(\w_+,\w_-)$ where $\w_{\pm}\in W_0\cap L_{\pm\lambda_0}$ and $\langle \w_+, \varphi_n^{k-1}(\w_-)\rangle=\frac{1}{2}$.
Moreover, ${\rm sign}(V_0)=(k,k)$;
\item $V\in \Delta_{2h+1}(\sqrt{-1}\lambda_0,-\sqrt{-1}\lambda_0)$, with $\lambda_0\in \R$: $W_0={\rm Span}_{\R}(\u,\v)$ where $\varphi_s(\u)=\lambda_0 \v$, $\varphi_s(\v)=-\lambda_0 \u$ and $\langle u, \varphi_n^{2h}(\v)\rangle=0$,  $\langle \u, \varphi_n^{2h}(\u)\rangle=\langle \v, \varphi_n^{2h}(\v)\rangle=\epsilon=\pm 1$.
Moreover, when $\epsilon=+1$ we have ${\rm sign}(V_0)=(2h+2,2h)$ if $h$ is even, while  ${\rm sign}(V_0)=(2h,2h+2)$ if $h$ is odd.
When $\epsilon=-1$ we have ${\rm sign}(V_0)=(2h,2h+2)$ if $h$ is even, while  ${\rm sign}(V_0)=(2h+2,2h)$ if $h$ is odd.
\item $V\in \Delta_{2h}(\sqrt{-1}\lambda_0,-\sqrt{-1}\lambda_0)$, with $\lambda_0\in \R$: $W_0={\rm Span}_{\R}(\u,\v)$ where $\varphi_s(\u)=\lambda_0 \v$, $\varphi_s(\v)=-\lambda_0 \u$ and $\langle \u, \varphi_n^{2h-1}(\v)\rangle=\pm\frac{1}{2}$.
Moreover, ${\rm sign}(V_0)=(2h,2h)$.
\end{enumerate}
\end{lem}
\begin{proof}In cases (i) and (ii), we can choose the basis vectors (as described in Lemma \ref{lem:block2}) in $W_0$. In case (iii) we can choose $\w_-=\overline{\w}_+$ and write $\w_+=\u-\sqrt{-1}\v$. In case (iii) we can choose $\w_-=\sqrt{-1}\overline{\w}_+$ and write $\w_+=\u-\sqrt{-1}\v$. Finally, the calculation of the signatures is a straightforward consequence of Lemma \ref{lem:block1}.\end{proof}
\noindent We will denote by $\Delta^{\pm}_{2h+1}(0)$, $\Delta^0_{k}(\lambda_0,-\lambda_0)$, and $\Delta_{k}^{\pm}(\sqrt{-1}\lambda_0,-\sqrt{-1}\lambda_0)$ the families of the real parts $V_0$ corresponding respectively to  the conditions (i), (ii), and the condition (iii) for $k=2h+1$ or (iv) for $k=2h$. In particular, note that if $k$ is even, $\Delta_{k}^-(\sqrt{-1}\lambda_0,-\sqrt{-1}\lambda_0)=\Delta_{k}^{+}(-\sqrt{-1}\lambda_0,\sqrt{-1}\lambda_0)$, so that we can set the sign to be a $+$ after reordering the eigenvalues.\\ \\
Finally, consider the case of an indecomposable complex block $V$ such that $\overline{V}\neq V$: in this case we have $\overline{V}\cap V=0$ and $V\in \Delta_k(\lambda,-\lambda)$ with $\lambda\notin \R\cup\sqrt{-1}\R$. In this case, $VV:=V\oplus \overline{V}$ admits a real part $VV_0$ and the subspace $WW:=W\oplus \overline{W}$ admits a real part $WW_0$, such that
$$VV_0=WW_0\oplus \varphi_n(WW_0)\oplus\dots \oplus\varphi^{k-1}_n(WW_0)$$
and $VV_0$ is an indecomposable real block, with semisimple part $WW_0$.
\begin{lem}Let $\lambda=a+\sqrt{-1}b$, where $a\neq 0\neq b$. Then
$WW_0={\rm Span}_{\R}\{\u_+,\u_-,\v_+,\v_-\}$ such that
$$\varphi_s(\u_{\pm})=\pm(a\u_{\pm}+b\v_{\pm}), \qquad \varphi_s(\v_{\pm})=\pm(-b\u_{\pm}+a\v_{\pm})$$
and
$$\langle \u_{\pm},\varphi_n^{k-1}(\u_{\pm})\rangle=\langle \v_{\pm},\varphi_n^{k-1}(\v_{\pm})\rangle=\langle \u_{\pm},\varphi_n^{k-1}(\v_{\pm})\rangle=\langle \u_{\pm},\varphi_n^{k-1}(\v_{\mp})\rangle=0$$
$$\langle \u_{+},\varphi_n^{k-1}(\u_-)\rangle=\frac{1}{2}, \qquad \langle \v_{+},\varphi_n^{k-1}(\v_-)\rangle=-\frac{1}{2}.$$
Moreover, ${\rm sign}(VV_0)=(2k,2k)$.
\end{lem}
\begin{proof}
Calculation analogous to the proof of Lemma \ref{lem:block3}.
\end{proof}
\noindent The families of the real subspaces $VV_0$ described above are denoted  $\Delta_k^0(\lambda,-\lambda,\overline{\lambda},-\overline{\lambda})$. We refer to the families 
\begin{equation}
\Delta^{\pm}_{2h+1}(0)\, ,\quad \Delta^0_{k}(\lambda_0,-\lambda_0)\, , \quad  \Delta_{k}^{\pm}(\sqrt{-1}\lambda_0,-\sqrt{-1}\lambda_0)\, , \textrm{ and }\quad  \Delta_k^0(\lambda,-\lambda,\overline{\lambda},-\overline{\lambda})
\end{equation} 
as \emph{indecomposable real types}. From the above discussion follows the
\begin{prop}Let $\varphi:\T M\rightarrow \T M$ be a skew-symmetric endomorphism. Then, at any point $p\in M$, the generalized tangent space $\T_p M$ splits into an orthogonal sum of real blocks, each of  type either $\Delta^{\pm}_{2h+1}(0)$, $\Delta^0_{k}(\lambda_0,-\lambda_0)$, $\Delta_{k}^{\pm}(\sqrt{-1}\lambda_0,-\sqrt{-1}\lambda_0)$ or $\Delta_k^0(\lambda,-\lambda,\overline{\lambda},-\overline{\lambda})$. 
\end{prop}

\begin{example} If $\varphi$ is a generalized almost complex structure, its real block decomposition is fiberwise of type
$m\cdot\Delta^+_{1}(\sqrt{-1},-\sqrt{-1})\oplus m\cdot\Delta^-_{1}(\sqrt{-1},-\sqrt{-1})$, where $2m=n$. More generally, if $\varphi$ is a generalized $f$-structure, the real block decomposition is of type
$m_1\cdot\Delta^+_{1}(\sqrt{-1},-\sqrt{-1})\oplus m_2\cdot\Delta^-_{1}(\sqrt{-1},-\sqrt{-1})\oplus n_1\cdot \Delta^+_{1}(0)\oplus n_2\cdot \Delta^-_{1}(0)$, where $2m_1+n_1=n=2m_2+n_2$.
\end{example}

\begin{example}\label{ex:54}
Suppose $M$ is 3-dimensional, fix $p\in M$, fix a basis $\{v_1,v_2,v_3\}$ of $T_pM$, and dual basis given by $\{\alpha_1,\alpha_2,\alpha_3\}$. Consider the endomorphism $\varphi:\T_pM\rightarrow \T_p M$ such that $\varphi(v_1)=v_2$, $\varphi(v_2)=v_3+\alpha_3$, $\varphi(v_3)=-\alpha_2$,
$\varphi(\alpha_1)=0$, $\varphi(\alpha_2)=-\alpha_1$, and $\varphi(\alpha_3)=-\alpha_2$. Then $\varphi$
has a decomposition of real type $\Delta_5^+(0)\oplus\Delta_1^-(0)$ and minimal polynomial $P(x)=x^5$.
\end{example}

\begin{example}\label{ex:55}
Continuing with the notation of Example \ref{ex:54}, suppose now that $\varphi:\T_pM\rightarrow \T_p M$ is defined by
$\varphi(v_1)=v_1-2\alpha_2$, $\varphi(v_2)=-v_2+2\alpha_1$, $\varphi(v_3)=\alpha_3$,
$\varphi(\alpha_1)=-\alpha_1$, $\varphi(\alpha_2)=\alpha_2$, and $\varphi(\alpha_3)=v_3$. Then $\varphi$
has decomposition of real type $\Delta_2^0(1,-1)\oplus\Delta_1^0(1,-1)$ and minimal polynomial $P(x)=(x^2-1)^2$.
\end{example}

\begin{example}\label{ex:56}
Again in the notation of Example \ref{ex:54}, suppose that  $\varphi:\T_pM\rightarrow \T_p M$ is such that
$\varphi(v_1)= v_2-\alpha_2$, $\varphi(v_2)= \alpha_1-\alpha_3$,  $\varphi(v_3)=v_2+\alpha_2$, 
$\varphi(\alpha_1)=0$, $\varphi(\alpha_2)=-\alpha_1-\alpha_3$, and $\varphi(\alpha_3)=0$. Then $\varphi$
has decomposition of type $\Delta^+_3(0)\oplus \Delta^-_3(0)$ and minimal polynomial $P(x)=x^3$.
\end{example}

\begin{rem} If $M$ is 3-dimensional and orientable, then it is parallelizable \cite{BenedettiLisca18}. Interpreting $\{v_1,v_2,v_3\}$ as a global frame and $\{\alpha_1,\alpha_2,\alpha_3\}$ as the associated dual frame, Examples \ref{ex:54}-\ref{ex:56} apply to $M$.
\end{rem}
\section{Weak generalized Nijenhuis operators}\label{sec:4}

In this section we investigate the compatibility of generalized polynomial structures with the Dorfman bracket or, equivalently, with the de Rham operator. Our main finding is that the semisimple part of a generalized polynomial structure controls the Dorfman involutivity of the generalized eigenbundles and the decomposition of the de Rham operator induced by the associated multigrading of $\Omega_M$. 

\begin{mydef}
Let $\varphi$ be an endomorphism of $\T M$. The {\it Courant-Nijenhuis torsion of $\varphi$} is defined as 
\begin{equation}
\mathcal T_\varphi(\x,\y)=\varphi^2\lbra \x,\y\rbra +\lbra \varphi(\x),\varphi(\y)\rbra - \varphi(\lbra \varphi(\x),\y\rbra +\lbra \x,\varphi(\y)\rbra)
\end{equation}
for all $\x,\y\in \T M$. Similarly, we define the {\it shifted Courant-Nijenhuis torsion of $\varphi$} as
\begin{equation}\label{eq:12}
    \mathcal S_\varphi(\x,\y)=\mathcal T_\varphi(\varphi(\x),\y)+ \mathcal T_\varphi(\x,\varphi(\y))
\end{equation}
for all $\x,\y\in \T M$.
\end{mydef}

\begin{mydef}[\cite{BoualemBrouzet12}] A {\it generalized Nijenhuis operator} is an endomorphism $\varphi$ of $\T M$ with vanishing Courant-Nijenhuis torsion. A {\it weak generalized Nijenhuis operator} is an endomorphism of $\T M$ with vanishing shifted Courant-Nijenhuis torsion.
\end{mydef}

\begin{rem}
If follows from \eqref{eq:12} that every generalized Nijenhuis operator is also a weak generalized Nijenhuis operator.
\end{rem}

\begin{example}[\cite{Gualtieri11}]\label{ex:52}
Let $\varphi$ be a generalized almost complex structure. Then $\varphi$ is a generalized complex structure if and only if it is a generalized Nijenhuis operator.
\end{example}

\begin{example}[\cite{BlagaCrasmareanu14}]\label{ex:53}
Let $\varphi$ be a weak generalized almost tangent structure. Then by definition $\varphi$ is a weak generalized tangent structure if and only if it is a generalized Nijenhuis operator (and not merely a weak generalized Nijenhuis operator, as the terminology might suggest).  
\end{example}

\begin{rem}\label{rem:54}
Let $\varphi$ be a non-zero semisimple generalized polynomial structure. Then 
\begin{equation}
    \mathcal T_\varphi(\x,\y)=(\varphi-\mu I)(\varphi-\nu I)\lbra \x,\y\rbra
\end{equation}
for all $\x\in L_\mu$, $\y\in L_\nu$.  Since ${\rm Ker}((\varphi-\mu I)(\varphi-\nu I))=L_\mu+L_\nu$ we conclude that $\varphi$ is a generalized Nijenhuis operator if and only if $\lbra L_\mu,L_\nu\rbra \subseteq L_\mu+L_\nu$. In particular, taking $\mu=\pm\lambda$ we see that the quasi split structure $E_\lambda$ is closed under the Dorfman bracket. Since $\varphi$ is non-zero and semisimple we have $\lambda\neq 0$. On the other hand, Remark \ref{rem:13} implies that $E_\lambda=\T M\otimes \mathbb C$ and thus $\varphi$ has precisely two eigenvalues: $\pm \lambda$. 
\end{rem}

\begin{rem}
Classically, integrability of polynomial structures is defined by requiring the vanishing of their Nijenhuis torsion. Thus, naively, one might be tempted to define a generalized polynomial structure $\varphi$ to be integrable if and only if $\varphi$ is a generalized Nijenhuis operator. Remark \ref{rem:54} shows that such a definition would be too restrictive, even for semisimple generalized polynomial structures with more than two eigenvalues. As noted in \cite{Kosmann-Schwarzbach19}, the necessity of relaxing the vanishing of the Nijenhuis torsion when working with general eigenbundles was already pointed in the classical case by Haantjes \cite{Haantjes55}. The connection between the shifted Courant-Nijenhuis torsion and the generalized analogues of the higher torsions introduced in \cite{Haantjes55} plays an important role in the reminder of this paper.
\end{rem}

\begin{theorem}\label{theorem:56}
Let $\varphi$ be a generalized polynomial structure. The following are equivalent:
\begin{enumerate}[1)]
\item the semisimple part of $\varphi$ is a weak generalized Nijenhuis operator;
\item $(\mu+\nu)\lbra L_\mu,L_\nu\rbra\subseteq L_\mu+L_\nu$ for all $\mu,\nu\in \Sigma(\varphi)$;
\item there exists a decomposition of the de Rham operator $d=\sum_{\lambda\in \Sigma(\varphi)} \delta_\lambda$ such that for every $\lambda\in \Sigma(\varphi)\setminus\{0\}$, $\delta_\lambda$ is of degree $1$ with respect to the $(L_\lambda,L_{-\lambda})$-grading and of degree $0$ with respect to the $(L_\mu,L_{-\mu})$-grading whenever $\mu\notin\{0,\pm\lambda\}$. 
\end{enumerate}
\end{theorem}

\begin{proof} Since the three statements only depend on the semisimple part of $\varphi$, we may (without loss of generality) assume that $\varphi$ is semisimple. With this assumption, the equivalence of 1) and 2) follows from the formula
\begin{equation}
    \mathcal S_\varphi(\x,\y)=(\mu+\nu)(\varphi-\mu I)(\varphi-\nu I)\lbra \x,\y\rbra\,,
\end{equation}
which is valid for all $\x,\in L_\mu$, $\y\in L_\nu$. Assume that a decomposition of the de Rham operator as in 3) is given. Viewed as operators acting on $\Omega_M\otimes \mathbb C$, sections of $L_\lambda$ are of degree $-1$ with respect to the $(L_\lambda,L_{-\lambda})$-grading. Since the generalized tangent bundle is closed under the Dorfman bracket, given $\x\in L_\mu$, $\y\in L_\nu$ such that $\mu+\nu\neq 0$ we have 
\begin{equation}
    \lbra \x, \y \rbra = \sum_{\lambda \in \Sigma(\varphi)}\lbra \x,\y\rbra_{\delta_\lambda} = \lbra \x,\y\rbra_{\delta_\mu} + \lbra \x,\y\rbra_{\delta_\nu}\in L_\mu+L_\nu\, , 
\end{equation}
which implies 2). We are left to prove that 2) implies 30). The statement is trivial if $\varphi$ is nilpotent so we may use induction on the number of non-zero eigenvalues. Suppose $\lambda\in \Sigma(\varphi)\setminus\{0\}$. In particular, this implies that $L_{\pm\lambda}$ is closed under the Dorfman bracket. By Lemma \ref{lem:23}, $d$ has component of degree $1$, $-1$, and $0$ with respect to the $(L_\lambda,L_{-\lambda})$-grading. Define $\delta_{\pm\lambda}$ to be the components of degree $\pm 1$ and let $\delta=d-\delta_\lambda-\delta_{-\lambda}$. If $\pm \lambda$ are the only non-zero eigenvalues then we may set $\delta_0=\delta$ and the Theorem is proved. If there is another eigenvalue $\mu\notin\{0,\pm \lambda\}$, then using the closure of $L_\mu$ under the Dorfman bracket and Lemma \ref{lem:23} again we may decompose $\delta$ into component of degree $1$, $-1$, $0$ with respect to $(L_\mu,L_{-\mu})$-grading and, if necessary, iterate. Since the order in which the non-zero eigenvalues are selected is irrelevant, we are left to show that $\delta_\lambda$ is of degree $0$ with respect to the $(L_\mu,L_{-\mu})$-grading whenever $\mu\notin\{0,\pm \lambda\}$. To see this, we note that $\delta_\lambda$ is of degree $1$ with respect to the $(L_\lambda,L_{-\lambda})$-grading (which commutes with the $(L_\mu,L_{-\mu})$-grading). Together with 2), this shows that the the assumptions of Lemma \ref{lem:24} are met proving that $\delta_\lambda$ does indeed preserve the $(L_\mu,L_{-\mu})$-grading. 
\end{proof}

\begin{rem}\label{rem:61}
Let $\varphi$ be a generalized polynomial structure with exactly two eigenvalues. This condition is equivalent to requiring that the minimal polynomial of $\varphi$ is of the form $P(x)=(x^2+c)^N$ for some non-zero real number $c$. Then the semisimple part $\varphi_s$ of $\varphi$ is a generalized Nijenhuis operator if and only if $\varphi_s$ is a weak generalized Nijenhuis operator if and only if $L_{\sqrt{-c}}$ and $L_{-\sqrt{-c}}$ are both closed under the Dorfman bracket. In the special case $N=c=1$, the equivalent conditions of Theorem \ref{theorem:56} are also equivalent to the statement that the generalized almost complex structure $\varphi=\varphi_s$ is integrable. Furthermore, in this case $\delta_{\sqrt{-1}}=\partial$ and $\delta_{-\sqrt{-1}}=\overline \partial$.   
\end{rem}

\begin{example}\label{ex:58}
Let $\varphi$ be a generalized polynomial structure with exactly three eigenvalues, so that its minimal polynomial is of the form $P(x)=x^{N_1}(x^2+c)^{N_2}$ for some non-zero real number $c$. Then $\varphi_s$ is a generalized Nijenhuis operator if and only if   
\begin{equation}\label{eq:15}
\lbra L_{\pm\sqrt{-c}},L_{\pm\sqrt{-c}}\rbra \subseteq L_{\pm\sqrt{-c}}\quad \textrm { and } \quad \lbra L_0,L_{\pm\sqrt{-c}}\rbra \subseteq L_0\oplus L_{\pm\sqrt{-c}}\,.
\end{equation}
Let us further specialize to the case $N_1=N_2=c=1$ i.e.\ to the case in which $\varphi=\varphi_s$ is a generalized $F$-structure. If $\varphi$ is a weak generalized Nijenhuis operator, then the first condition in \eqref{eq:15} shows that $\varphi$ is a generalized CRF structure. Generalized $F$-structure satisfying both conditions in \eqref{eq:15} are sometimes referred to \cites{PoonWade11,AldiGrandini15,AldiGrandini17} as {\it strongly integrable}. Furthermore, in this case the decomposition of the de Rham operator is the one discussed in \cite{AldiGrandini17}.
\end{example}

\begin{rem} Let $\varphi$ be a generalized polynomial structure. Assume $\varphi$ is a weak generalized Nijenhuis operator and let $d=\sum_\lambda \delta_\lambda$ be the decomposition of the de Rham operator whose existence is guaranteed by Theorem \ref{theorem:56}. Setting equal to zero the components of definite multidegree in $d^2=0$, we obtain 
$(\mu+\nu)[\delta_\mu,\delta_\nu]=0$ for every $\lambda,\mu\in \Sigma(\varphi)$. In particular, the components of the de Rham operator corresponding to non-zero eigenvalues are differentials. We also have the additional equation
\begin{equation}
    \delta_0^2=-\frac{1}{2}\sum_{\lambda\in \Sigma_+(\varphi)} [\delta_\lambda,\delta_{-\lambda}]\,.
\end{equation}
\end{rem}

\begin{rem}
In the classical case, Vanzurova \cite{Vanzurova98} restated and studied the integrability of semisimple polynomial structures in terms of a decomposition of the de Rham operator.
\end{rem}

\section{Minimality}\label{sec:5}

The goal of this section is to introduce a novel notion of compatibility between generalized polynomial structures and the de Rham operator that depends on both their semisimple and nilpotent parts.  

\subsection{The minimal torsion and the Courant tensor}

Let $\varphi$ be a generalized polynomial structure on $M$ with minimal polynomial $P$ and let $\widetilde \varphi$ be the lift of $\varphi$ to $\mathcal E_M$, whose existence is guaranteed by Proposition \ref{prop:9}. We define the {\it minimal operator of $\varphi$} as $\delta_\varphi=P(\ad_{\widetilde \varphi})(d)$, where $d$ is the de Rham operator.

\begin{example}
Let $\varphi$ be a generalized almost complex structure. Then $P(x)=x^2+1$ and $\delta_\varphi=[\widetilde \varphi,[\widetilde\varphi,d]]+d=-\lbra \widetilde \varphi,\widetilde\varphi\rbra+d$. 
\end{example}

\begin{lem}\label{lem:40}
The minimal operator of a generalized polynomial structure is $\Omega_M^0$-linear.
\end{lem}

\begin{proof}
Since $[\ad_{\widetilde \varphi},f]=0$ for all $f\in \Omega_M^0$, then $[\delta_\varphi,f]=P(\ad_{\widetilde\varphi})(df)=P(\varphi)(df)=0$
\end{proof}

\begin{mydef}
The {\it minimal torsion of $\varphi$}, denoted by $\mathcal M_\varphi:\BigWedge^2 \T M\to \T M$, is by definition the derived bracket with respect to the minimal operator of $\varphi$ i.e.\ $\mathcal M_\varphi(\x,\y)=\lbra \x,\y\rbra_{\delta_\varphi}$ for all $\x,\y\in \T M$. The {\it Courant tensor of $\varphi$} is the map $\mathcal C_\varphi:\BigWedge^3 \T M\to \Omega_M^0$ defined by the formula $\mathcal C_\varphi(\x,\y,\z)=\langle \mathcal M_\varphi(\x,\y),\z\rangle$ for all $\x,\y,\z\in \T M$.
\end{mydef}

\begin{rem}
More generally, if $\delta\in \mathcal E_M$ is an operator such that $\lbra \T M, \T M\rbra_\delta \subseteq \T M$ one can introduce the $\Omega_M^0$-valued ternary operation
\begin{equation}
    T_\delta(\x,\y,\z)=\langle \lbra \x,\y\rbra_\delta ,\z \rangle 
\end{equation}
for all $\x,\y,\z\in \T M$. In the particular case $\delta=d$ one obtains the $\mathbb R$-trilinear map $T(\x,\y,\z)=\langle \lbra \x,\y\rbra ,\z \rangle$ introduced in \cite{Courant90}. 
\end{rem}

\begin{rem}\label{rem:68}
The definition of the minimal operator depends on the particular lift of $\varphi$ to $\mathcal E_M$ obtained by imposing condition 3) in Proposition \ref{prop:9}. If a different lift $\widetilde \varphi + g$ is chosen, for some function $g\in \Omega_M^0$, then $P(\ad_{\widetilde \varphi+g})(d)=\delta_{\varphi}-P'(\varphi)(dg)$. In particular, the definition of the minimal operator depends on the choice of lift only up to addition of a section of $\T M$. Hence, by Lemma \ref{lemma:7}, the definition of the minimal torsion (and thus of the Courant tensor) of $\varphi$ is independent on condition 3) in Proposition \ref{prop:9}. 
\end{rem}

\begin{rem}
Combining Lemma \ref{lem:8} and Lemma \ref{lem:40}, we see that $\mathcal M_\varphi$ is indeed skew-symmetric and $\Omega_M^0$-bilinear. An analogous calculation shows that $\mathcal C_\varphi$ is totally anti-symmetric and $\Omega_M^0$-trilinear. 
\end{rem}

\begin{prop}\label{prop:44}
Let $\varphi$ be a generalized polynomial structure on $M$. Then the following conditions are equivalent:
\begin{enumerate}[1)]
\item $\delta_\varphi\in \T M$; 
\item the minimal torsion of $\varphi$ vanishes;
\item the Courant tensor of $\varphi$ vanishes.
\end{enumerate}
\end{prop}

\begin{proof}
The equivalence between 2) and 3) is a consequence of the non-degeneracy of the tautological inner product. The equivalence of 1) and 2) is a consequence of Lemma \ref{lemma:7}.
\end{proof}

\begin{mydef}
A generalized polynomial structure is {\it minimal} if the equivalent conditions 1)-3) in Proposition \ref{prop:44} are satisfied.
\end{mydef}

\begin{prop}\label{prop:46}
Let $\varphi$ be a generalized polynomial structure with minimal polynomial $P(x)=a_Nx^N+\cdots +a_1x+a_0$. Then
\begin{equation}\label{eq:7}
    \mathcal C_\varphi(\x,\y,\z)=\sum_{i=0}^N \sum_{i_1+i_2+i_3=i} a_i (-1)^i\binom{i}{i_1,i_2,i_3} T(\varphi^{i_1}(\x),\varphi^{i_2}(\y),\varphi^{i_3}(\z))
\end{equation}
for all $\x,\y,\z\in \T M$.
\end{prop}

\begin{proof}
Assume $\delta\in \mathcal E_M$ is odd and such that $\lbra \T M,\T M\rbra_\delta \subseteq \T M$. Then, for all $\x,\y,\z\in \T M$, the operator 
\begin{equation}\label{eq:39}
(\ad_\z \ad_\y\ad_\x \ad_{\widetilde \varphi}  + \ad_{\varphi(\z)}\ad_\y\ad_\x + \ad_\z\ad_{\varphi(\y)}\ad_\x + \ad_\z\ad_\y\ad_{\varphi(\x)}-\ad_{\widetilde \varphi} \ad_\z \ad_\y\ad_\x)(\delta) 
\end{equation}
is equal to zero. The last term of \eqref{eq:39} vanishes because $(\ad_\z \ad_\y\ad_\x) (\delta)\in \Omega_M^0$ and $[\widetilde \varphi,\Omega_M^0]=0$. Hence
\begin{equation}\label{eq:8}
T_{[\widetilde \varphi,\delta]}(\x,\y,\z)+T_\delta (\varphi(\x),\y,\z)+T_\delta (\x,\varphi(\y),\z)+ T_\delta(\x,\y,\varphi(\z))=0\,.
\end{equation}
Iterating this step, we find 
\[
T_{\ad^i_{\widetilde \varphi}(d)}(\x,\y,\z) = (-1)^i\sum_{i_1+i_2+i_3=i} \binom{i}{i_1,i_2,i_3}T(\varphi^{i_1}(\x),\varphi^{i_2}(\y),\varphi^{i_3}(\z))\, ,
\]
from which the result easily follows.
\end{proof}

\begin{rem}\label{rem:PolynomialAction}
Let $\varphi$ be an endomorphism of $\T M$. Following \cite{Kosmann-Schwarzbach19}, we consider the action of the polynomial ring $\mathbb C[u,v,w]$ on the space of local linear operators $\zeta:\T M\otimes \T M\to \T M$ defined by $(u\cdot \zeta)(\x,\y)=\varphi(\zeta(\x,\y))$, $(v\cdot \zeta)(\x,\y)=\zeta(\varphi(\x),\y)$, and $(w\cdot \zeta)(\x,\y)=\zeta(\x,\varphi(\y))$. If $\zeta_{CD}$ is the Dorfman bracket, then Proposition \ref{prop:46} can be equivalently restated as $\mathcal M_\varphi=P(u-v-w)\cdot \zeta_{CD}$. 
Note that that the above  action is equivalent to the action of $\C[\varphi_1, \varphi_2, \varphi_3]$ on the sheaf $\T M^{\otimes 3}$ determined by
$\varphi_1 \cdot = \varphi \otimes {\rm Id}\otimes {\rm Id}, \quad \varphi_2 \cdot = {\rm Id}\otimes\varphi \otimes {\rm Id}, \quad \varphi_3 \cdot = {\rm Id} \otimes {\rm Id}\otimes\varphi$
so that, for all $\zeta:\T M\otimes \T M\to \T M$,
$\langle u^iu^ju^k\zeta(\x,\y),\z\rangle = f((-\varphi_3)^i\varphi_1^j\varphi_2^k(\x\otimes\y\otimes \z))$
where $f(\x\otimes\y\otimes\z)=\langle\zeta(\x,\y),\z\rangle$.
Using this latter action, Proposition \ref{prop:46} reads $\mathcal C_\varphi=T\circ P(-\varphi_1-\varphi_2-\varphi_3)$.

\end{rem}
\begin{example}\label{ex:47}
Let $\varphi$ be a generalized almost complex structure. Then
\begin{equation}
    \mathcal C_\varphi(\x,\y,\z)=2(-T(\x,\y,\z)+T(\varphi(\x),\varphi(\y),\z)+T(\varphi(\x),\y,\varphi(\z))+T(\x,\varphi(\y),\varphi(\z)))
\end{equation}
for all $\x,\y,\z\in \T M$. Using the skew-symmetry of $\varphi$ and non-degeneracy of the tautological inner product this is equivalent to $\mathcal M_\varphi=2 \mathcal T_\varphi$, where $\mathcal T_\varphi$ is the Courant-Nijenhuis torsion. In particular $\varphi$ is minimal if and only if it is a generalized complex structure.
\end{example}

\begin{example}\label{ex:48}
Let $\varphi$ be a generalized F-structure. The Courant torsion of $\varphi$ can be calculated using Proposition \ref{prop:46}. From there, using skew-symmetry and non-degeneracy as in Example \ref{ex:47}, we obtain $\mathcal M_\varphi=-3\mathcal S_{\varphi}$, where $\mathcal S_\varphi$ is the shifted Courant-Nijenhuis torsion of $\varphi$. Hence, thanks to Example \ref{ex:58}, we conclude that $\varphi$ is minimal if and only if it is strongly integrable.
\end{example}

\begin{prop}\label{prop:52}
Let $\varphi$ be a generalized polynomial structure with minimal polynomial $P$ of degree $N$. Consider Jordan chains $\{\x_\alpha\}\subseteq L_\lambda$, $\{\y_\beta\}\subseteq L_\mu$, and $\{\z_\gamma\}\subseteq L_\nu$, so that $\varphi(\x_\alpha)=\lambda \x_\alpha+\x_{\alpha-1}$, $\varphi(\y_\beta)=\mu \x_\beta+\x_{\beta-1}$, and $\varphi(\z_\gamma)=\nu \x_\gamma+\x_{\gamma-1}$. Then
\begin{equation}\label{eq:11}
    \mathcal C_\varphi (\x_\alpha,\y_\beta,\z_\gamma)=(-1)^NP(\lambda+\mu+\nu+S)T(\x_\alpha,\y_\beta,\z_\gamma)\, ,
\end{equation}
where $S$ is the shift operator defined by the formula
\begin{equation}\label{eq:43}
    S T(\x_\alpha,\y_\beta,\z_\gamma)=T(\x_{\alpha-1},\y_\beta,\z_\gamma)+ T(\x_{\alpha}, \y_{\beta-1},\z_\gamma)+T(\x_\alpha,\y_\beta,\z_{\gamma-1})\,.
\end{equation}
\end{prop}

\begin{proof}
Consider a factorization $P(x)=\prod_{i=1}^N (x-\omega_i)$ with $\omega_i\in \Sigma(\varphi)$ not necessarily distinct. Using \eqref{eq:8} we have
\begin{equation}
    T_{(\ad_{\widetilde\varphi}-\omega_i)(\delta)}(\x_\alpha,\y_\beta\,\z_\gamma)=-(\lambda+\mu+\nu+\omega_i+S)T_\delta(\x_\alpha,\y_\beta,\z_\gamma)\, .
\end{equation}
Keeping into account that $\Sigma(\varphi)$ is symmetric with respect to $0\in \mathbb C$, and iterating over all factors of $P$, yields the result.
\end{proof}

\begin{cor}\label{cor:88}
Let $\varphi$ be a generalized polynomial structure with minimal polynomial $P(x)$. If $\x\in {\rm Ker}((\varphi-\lambda I)^p)$ and $\y\in{\rm Ker}((\varphi-\mu I)^q)$ then
\begin{equation}
    P^{p+q-1}((\lambda+\mu) I-\varphi)(\lbra \x,\y\rbra)=0\,
\end{equation}
\end{cor}

\begin{proof}
Let $\mathbf S$ be the shift operator defined on derived brackets of Jordan chains by  $\mathbf S \lbra \mathbf x_\alpha,\mathbf y_\beta\rbra_\delta=\lbra \mathbf x_{\alpha-1},\mathbf y_\beta\rbra_\delta+\lbra \mathbf x_\alpha,\mathbf y_{\beta-1}\rbra_\delta$. Then we have the Taylor expansion
\begin{equation}
    P^r((\lambda+\mu)I-\varphi+\mathbf S)=\sum_{s=0}^\infty (P^r)^{(s)}((\lambda+\mu) I -\varphi)\frac{\mathbf S^s}{s!}
\end{equation}
for every integer $r\ge 1$. The result then follows upon using induction on $p+q$ and the fact that $P$ divides $(P^r)^{(s)}$ for every $r>s$.
\end{proof}

\begin{prop}\label{prop:87}
Let $\varphi$ be a minimal generalized polynomial structure. Then the semisimple and nilpotent parts of $\varphi$ are both minimal.
\end{prop}

\begin{proof}
It follows from $\eqref{eq:44}$ that the $\varphi_s=Q_s(\varphi)$ for some polynomial $Q_s$, which also implies $\varphi_n=Q_n(\varphi)$ with $Q_n(x)=x-Q_s(x)$. Let $P_s$ be the minimal polynomial of $\varphi_s$. Then $0=P_s(\varphi_s)=P_s(Q_s(\varphi))$ and thus $P_s(Q_s(x))$ is divisible by the minimal polynomial of $\varphi$. Hence \begin{equation}
    P_s(\ad_{\widetilde {\varphi_s}})(d)=P_s(Q_s(\ad_{\widetilde\varphi}))(d)
\end{equation}
is a section of $\mathbb T M$, proving that $\varphi_s$ is minimal. A completely analogous argument shows that $\varphi_n$ is minimal.
\end{proof}

\begin{cor}\label{cor:89}
Let $\varphi$ be a generalized polynomial structure such that all the eigenvalues have the same multiplicity. Then $\varphi$ is minimal if and only if its semisimple and nilpotent parts are both minimal.
\end{cor}

\begin{proof}
If $\varphi$ is minimal, then $\varphi_s$ and $\varphi_n$ are both minimal by Proposition \ref{prop:87}. For the converse, let $k$ be the common multiplicity of all eigenvalues and let $P_s(x)$ be the minimal polynomial of $\varphi_s$ so that $P(x)=P_s^k(x)$ is the minimal polynomial of $\varphi$. If $\varphi_n$ is minimal, then $\ad_{\widetilde{\varphi_n}}^j(d)\in \T M$ for all $j\ge k$ and thus, upon taking the Taylor expansion,
\begin{equation}
    P(\ad_{\widetilde\varphi})(d)-P(\ad_{\widetilde{\varphi_s}})(d)-\cdots-P^{(k-1)}(\ad_{\widetilde{\varphi_s}})\ad^{k-1}_{\widetilde{\varphi_n}}(d)\in \T M.
\end{equation}
On the other hand, $P_s(x)$ divides $P^{(i)}(x)$ for all $0\le i<k$ which concludes the proof.
\end{proof}

\begin{prop}\label{prop:88}
Let $\varphi$ be a semisimple generalized polynomial structure. If $\varphi$ is a weak generalized Nijenhuis operator, then $\varphi$ is minimal.
\end{prop}

\begin{proof}Assume $\varphi$ has minimal polynomial $P(x)$ of degree $N$.
Due to semisimplicity, the shift operator $S$ in \eqref{eq:11} acts trivially. For every $\x\in L_\lambda$, $\y\in L_\mu$, and $\z\in L_\nu$ we have
\begin{equation}\label{eq:47}
    \mathcal C_{\varphi}(\x,\y,\z)=(-1)^N P(\lambda +\mu+\nu)T(\x,\y,\z)
\end{equation}
where $P(\lambda +\mu+\nu)T(\x,\y,\z)$ is the ordinary product of real numbers. By Theorem \ref{theorem:56}, $T(\x,\y,\z)=0$ unless $(\lambda+\mu)(\lambda+\nu)(\mu+\nu)=0$ in which case $P(\lambda +\mu+\nu)=0$. Hence the Courant tensor vanishes identically and $\varphi$ is minimal.  
\end{proof}

\subsection{Higher Courant-Nijenhuis torsions}

Let $\varphi$ be an operator acting on $\mathbb T M$. Following \cite{Kosmann-Schwarzbach19} we define the {\it higher Courant-Nijenhuis torsions} $\mathcal T_\varphi^{(n)}$, for any positive integer $n$, as follows. If $n=1$ set $\mathcal T_\varphi^{(1)}=\mathcal T_\varphi$. If $n\ge 1$ we define $\mathcal T_\varphi^{(n)}$ recursively by the formula
\begin{equation}\label{eq:17}
    \mathcal T_\varphi^{(n+1)}(\x,\y)=\varphi^2 \mathcal T_\varphi^{(n)}(\x,\y)+\mathcal T_\varphi^{(n)}(\varphi(\x),\varphi(\y))-\varphi \mathcal S_\varphi^{(n)}(\x,\y)
\end{equation}
for all $\x,\y\in \T M$ where, for every $n\ge 1$,
\begin{equation}
\mathcal S_\varphi^{(n)}(\x,\y)=\mathcal T_\varphi^{(n)}(\varphi(\x),\y)+\mathcal T_\varphi^{(n)}(\x,\varphi(\y))
\end{equation}
is the $n$-th {\it shifted higher Courant-Nijenhuis torsion} of $\varphi$. It is convenient to extend this definition and set $\mathcal T_\varphi^{(0)}(\x,\y)=\lbra \x,\y\rbra$ and $\mathcal S^{(0)}_\varphi (\x,\y)=\lbra \varphi(\x),\y\rbra+\lbra \x, \varphi(\y)\rbra$ for all $\x,\y\in \T M$.

\begin{rem}\label{rem:65}
Let $f:TM\to TM$ be a polynomial structure and let $\varphi=f\oplus (-f^T)$ the associated generalized polynomial structure. Then $\mathcal T_\varphi^{(m)}(X,Y)=\tau_f^{(m)}(X,Y)$ for all $X,Y\in TM$, where $\tau_f^{(m)}$ is the $m$-th higher Nijenhuis torsion of $f$ as defined in \cite{TempestaTondo18a}. In particular, $\tau^{(1)}_f$ and $\tau^{(2)}_f$ are respectively the ordinary Nijenhuis torsion and the Haantjes torsion of $f$. 
\end{rem}

\begin{example}
If $n=2$ we have a Courant algebroid analogue of the Haantjies tensor
\begin{align*}
    \mathcal T_\varphi^{(2)}(\x,\y)=&\varphi^4(\lbra \x,\y\rbra) + \lbra \varphi^2(\x),\varphi^2(\y)\rbra + \varphi^2(\lbra \varphi^2(\x),\y\rbra +\lbra \x,\varphi^2(\y)\rbra+4\lbra \varphi(x),\varphi(\y)\rbra) \\
    &-2\varphi(\varphi^2(\lbra \varphi(\x),\y\rbra+\lbra \x,\varphi(\y)\rbra)+\lbra \varphi^3(\x),\y\rbra+\lbra \x,\varphi^3(\y)\rbra)
\end{align*}
for which we propose the terminology {\it Courant-Haantjies tensor}.
\end{example}

\begin{rem}\label{rem:56}
Using the notation introduced in Remark \ref{rem:PolynomialAction}, we have
\begin{equation}
    \mathcal T_\varphi^{(n)}=(u^2-uv-uw+vw)^n\cdot\zeta_{CD}=(u-v)^n(u-w)^n\cdot \zeta_{CD}\, .
\end{equation}
Therefore,
\begin{equation}
    \mathcal T_\varphi^{(n)}(\x,\y)=\sum_{i,j=0}^n (-1)^{i+j}\binom{n}{i}\binom{n}{j}\varphi^{i+j}\lbra \varphi^{n-i}(\x),\varphi^{n-j}(\y)\rbra
\end{equation}
for all $\x,\y\in \T M$. Similarly, from
\begin{equation}
    \mathcal S_\varphi^{(n+1)}=(u-v)^n(u-w)^n(u^2-uv-uw+vw)(v+w)\cdot \zeta_{CD}
\end{equation}
we obtain
\begin{equation}
    \mathcal S_\varphi^{(n+1)}(\x,\y)=\sum_{i,j=0}^{n}(-1)^{i+j}\binom{n}{i}\binom{n}{j}\varphi^{i+j}\mathcal S_\varphi(\varphi^{n-i}(\x),\varphi^{n-j}(\y))
\end{equation}
for all $\x,\y\in \T M$.
\end{rem}

\begin{theorem}\label{theorem:57}
Let $\varphi$ be a generalized polynomial structure with minimal polynomial $P(x)$. If $P(x)=a_0+a_2x^2+\cdots+a_{2N}x^{2N}$ then 
\begin{equation}\label{eq:20}
    \mathcal M_\varphi(\x,\y)=2\sum_{n=1}^{N} \sum_{m=1}^n \sum_{r=0}^{2n-2m} a_{2n} C_{n,m,r}\mathcal T_\varphi^{(m)}(\varphi^r(\x),\varphi^{2n-2m-r}(\y))
\end{equation}
for all $\x,\y\in \T M$, where
\begin{equation}\label{eq:21}
    C_{n,m,r}=\sum_{k=m}^n\sum_{q=2n-2m-r}^{2n-2k-r}(-1)^{q+r}4^{m-n}\binom{2n}{2k}\binom{k}{m}\binom{2k-2m}{2n-2m-q-r}\binom{2n-2k}{q}\,.
\end{equation}
If $P(x)=a_1x+a_3x^3+\cdots+a_{2N+1}x^{2N+1}$, then 
\begin{equation}\label{eq:22}
    \mathcal M_\varphi(\x,\y)=-\sum_{n=1}^{N} \sum_{m=1}^n \sum_{r=0}^{2n-2m} a_{2n+1} D_{n,m,r}\mathcal S_\varphi^{(m)}(\varphi^r(\x),\varphi^{2n-2m-r}(\y))
\end{equation}
for all $\x,\y\in \T M$, where
\begin{equation}\label{eq:23}
    D_{n,m,r}=\sum_{k=m}^n\sum_{q=2n-2m-r}^{2n-2k-r}(-1)^{q+r}4^{m-n}\binom{2n+1}{2k}\binom{k}{m}\binom{2k-2m}{2n-2m-q-r}\binom{2n-2k}{q}\,.
\end{equation}
\end{theorem}

\begin{proof} Let $\mathcal M_\varphi^{(n)}$ denote the derived bracket of the operator $(\ad_{\widetilde\varphi}^n+(-1)^n\ad_{\widetilde{\varphi^n}})(d)$. Assume $P(x)=a_0+a_2x^2+\cdots+a_{2N}x^{2N}$. Then, by Remark \ref{rem:PolynomialAction}, the quantity
\begin{equation}
   \mathcal M_\varphi(\x,\y)-\sum_{n=1}^N a_{2n} \mathcal M^{(2n)}_\varphi(\x,\y)
   \end{equation}
is equal to 
\begin{equation}
    a_0\lbra \x,\y\rbra+\sum_{n=1}^N a_{2n}( \varphi^{2n}(\lbra \x,\y\rbra)-\lbra\varphi^{2n}(\x),\y \rbra  -\lbra \x,\varphi^{2n}(\y)\rbra)=0\, .
\end{equation}
Using the notation of Remark \ref{rem:PolynomialAction} and setting $z=2u-v-w$, we have
\begin{align*}
    \mathcal M_\varphi^{(2n)} &= ((u-v-w)^{2n}+u^{2n}-v^{2n}-w^{2n})\cdot \zeta_{CD}\\
    & = \left( \left(\frac{z-v-w}{2} \right)^{2n}+\left(\frac{z+v+w}{2} \right)^{2n} - v^{2n}-w^{2n} \right) \cdot \zeta_{CD}\\
    &= \left(2^{1-2n}\sum_{k=0}^n \binom{2n}{2k}z^{2k}(v+w)^{2n-2k}-v^{2n}-w^{2n} \right) \cdot \zeta_{CD}\,.
\end{align*}
Setting $t=u^2-u(v+w)+v w$ we have $\mathcal T_\varphi^{(m)}=t^m\cdot \zeta_{CD}$. On the other hand, $z^2=4t+(v-w)^2$ from which we obtain
\begin{align*}
    \mathcal M_\varphi^{(2n)}&=\left(2^{1-2n}\sum_{k=0}^n\binom{2n}{2k}(v+w)^{2n-2k}\sum_{m=0}^k 4^mt^m(v-w)^{2k-2m}-v^{2n}-w^{2n}\right)\cdot \zeta_{CD}\\
    &=\left(2\sum_{m=0}^n 4^{m-n}t^m\sum_{k=m}^n\binom{2n}{2k}\binom{k}{m}(v-w)^{2k-2m}(v+w)^{2n-2k}-v^{2n}-w^{2n} \right) \cdot \zeta_{CD}\\
    &=\left(\sum_{m=1}^n 2^{2m-2n+1}t^m\sum_{k=m}^n\binom{2n}{2k}\binom{k}{m}(v-w)^{2k-2m}(v+w)^{2n-2k} \right) \cdot \zeta_{CD}\, ,
\end{align*}
where the last equality follows from the identity
\begin{equation}
    v^{2n}+w^{2n}=2^{1-2n}\sum_{k=0}^n\binom{2n}{2k}(v-w)^{2k}(v+w)^{2n-2k}\,.
\end{equation}
By expanding $(v-w)^{2k-2m}(v+w)^{2n-2k}$ and then summing over $n$ we obtain \eqref{eq:20} and \eqref{eq:21}. The proof of \eqref{eq:22} and \eqref{eq:23} is analogous: starting with
\begin{equation}
    \mathcal M_\varphi^{(2n+1)} = ((u-v-w)^{2n+1}-u^n+v^n+w^n)\cdot \zeta_{CD}
\end{equation}
one repeats the changes of variables as above and then, using the representation $\mathcal S_\varphi^{(m)}=t^m(v+w)\cdot \zeta_{CD}$, obtain the desired expansion.
\end{proof}

\begin{example}\label{ex:69}
By Remark \ref{rem:34}, if the minimal polynomial $P(x)$ is quadratic it has to be of the form $x^2+a_0$ for some real number $a_0$. Substituting into \eqref{eq:20} we obtain $\mathcal M_\varphi = 2\mathcal T_\varphi$, independently of $a_0$. We conclude that a generalized polynomial structure of degree 2 is minimal if and only if it is a generalized Nijenhuis operator. In particular, for $a_0=1$ we recover the the observation made in Example \ref{ex:47}.
\end{example}

\begin{example}\label{ex:87}
Similarly, if the minimal polynomial $P(x)$ is cubic it has to be of the form $x^3+a_1x$ for some real number $a_1$. Substituting into \eqref{eq:22} we obtain $\mathcal M_\varphi=-3\mathcal S_\varphi$ independently of $a_1$. We conclude that a generalized polynomial structure of degree 3 is minimal if and only if it is a weak generalized Nijenhuis operator.
\end{example}

\begin{cor}\label{cor:90}
Let $\varphi$ be generalized polynomial structure of degree $4$ with exactly two eigenvalues. Then $\varphi$ is minimal if and only its semisimple and nilpotent parts are both generalized Nijenhuis operators.  
\end{cor}

\begin{proof}
By assumption the minimal polynomial of $\varphi$ is of the form $(x^2+c)^2$ for some non-zero $c$. Hence the minimal polynomials of the semisimple and nilpotent part are, respectively, $x^2+c$ and $x^2$. In particular they are both quadratic and the statement follows by combining Corollary \ref{cor:89} and Example \ref{ex:69}.
\end{proof}

\begin{cor}
There is a canonical bijection between the set of minimal generalized polynomial structures with minimal polynomial $P(x)=(x^2+1)^2$ and the set of pairs $(\mathcal J_1,\mathcal J_2)$ consisting of a generalized complex structure $\mathcal J_1$ and a weak generalized tangent structure $\mathcal J_2$ such that $[\mathcal J_1,\mathcal J_2]=0$. 
\end{cor}

\begin{proof}
Given a generalized polynomial structure $\varphi$ with minimal polynomial $P(x)=(x^2+1)^2$, then $\mathcal J_1=\varphi_s$ is a generalized almost complex structure and $\mathcal J_2=\varphi_n$ is a weak generalized almost tangent structure. Conversely, if $\mathcal J_1$ and $\mathcal J_2$ are as above then $\varphi=\mathcal J_1+\mathcal J_2$ is a generalized polynomial structure with minimal polynomial $P(x)=(x^2+1)^2$. The result then follows from Corollary \ref{cor:90}. 
\end{proof}

\begin{cor}\label{cor:71}
Let $\varphi$ be a generalized polynomial structure with minimal polynomial $P(x)$ of degree at least 4. If $P(x)=a_0+a_2 x^2+\cdots+a_{2N-2}x^{2N-2}+x^{2N}$ then $\varphi$ is minimal if and only if
\begin{equation}
    \mathcal T_\varphi^{(N)}(\x,\y)=-\sum_{n=1}^{N-1}\sum_{m=1}^n\sum_{r=0}^{2n-2m}a_{2n}C_{n,m,r}\mathcal T_\varphi^{(m)}(\varphi^r(\x),\varphi^{2n-2m-r}(\y))
\end{equation}
for all $\x,\y\in \T M$, where the coefficients $C_{n,m,r}$ are defined by \eqref{eq:21}. If $P(x)=a_1x+a_3x^3+\cdots+a_{2N}x^{2N}+x^{2N+1}$ then $\varphi$ is minimal if and only if
\begin{equation}\label{eq:35}
    \mathcal S_\varphi^{(N)}(\x,\y)=-\frac{1}{2N+1}\sum_{n=1}^{N-1}\sum_{m=1}^n\sum_{r=0}^{2n-2m}a_{2n+1}D_{n,m,r}\mathcal S_\varphi^{(m)}(\varphi^r(\x),\varphi^{2n-2m-r}(\y))
\end{equation}
for all $\x,\y\in \T M$, where the coefficients $D_{n,m,r}$ are defined by \eqref{eq:23}.
\end{cor}

\begin{example}
If $P(x)$ is quartic, then $P(x)=x^4+a_2x^2+a_0$ for some real numbers $a_2$ and $a_0$. In this case the minimal torsion depends on $a_2$ but not on $a_0$. Moreover, $\varphi$ is minimal if and only if the Courant-Haantjes torsion satisfies
\begin{equation}
    \mathcal T^{(2)}_\varphi(\x,\y)=-2(\mathcal T_\varphi(\varphi(\x),\varphi(\y))+\mathcal T_\varphi(\varphi^2(\x),\y)+\mathcal T_\varphi(\x,\varphi^2(\y)))-a_2\mathcal T_\varphi(\x,\y)\,
\end{equation}
for all $\x,\y\in \T M$.
\end{example}

\begin{cor}\label{cor:74}
Let $\varphi$ be a generalized polynomial structure.  
\begin{enumerate}[1)]
\item If $\varphi$ is generalized Nijenhuis operator, then $\varphi$ is minimal.
\item If $\varphi$ is of odd degree and a weak generalized Nijenhuis operator, then it is minimal.
\end{enumerate}
\end{cor}

\begin{proof}
It follows from \eqref{eq:17} that $\mathcal T_\varphi=0$ implies $\mathcal T_\varphi^{(m)}=0$ for every $m\ge 1$. Then 1) follows from Theorem \ref{theorem:57}. The proof of 2) is similar and left to the reader.
\end{proof}

\begin{example}
Let $f:TM\to TM$ be a polynomial structure and let $\varphi=f\oplus (-f^T)$ be the associated generalized polynomial structure. Since the Dorfman bracket vanishes on $T^*M$ and preserves $TM$, then (using the anti-symmetry of the Courant tensor) we conclude that $\varphi$ is minimal if and only if $\mathcal C_\varphi(TM,TM,T^*M)=0$. Let us further assume assume that $f$ is integrable or, equivalently, that the Nijenhuis torsion of $f$ vanishes i.e.\ $\mathcal T_\varphi(TM,TM)=0$. Hence, invoking Theorem \ref{theorem:57}, the integrability of $f$ implies the minimality of $\varphi$. More generally, combining Remark \ref{rem:65} and Corollary \ref{cor:71} we obtain a remarkable characterization of minimality for generalized polynomial structures of the form $\varphi=f\oplus (-f^T)$ purely in terms of (ordinary) higher Nijenhuis torsions. Namely, if the minimal polynomial of $\varphi$ is $P(x)=a_0+a_2 x^2+\cdots+a_{2N-2}x^{2N-2}+x^{2N}$, then minimality is equivalent to the $N$-th higher Nijenhuis torsion being expressed as linear combination with constant coefficients of the lower order torsions according to the formula
\begin{equation}\label{eq:33}
    \tau_f^{(N)}(X,Y)=-\sum_{n=1}^{N-1}\sum_{m=1}^n\sum_{r=0}^{2n-2m}a_{2n}C_{n,m,r}\tau_f^{(m)}(f^r(X),f^{2n-2m-r}(Y))
\end{equation}
for all $X,Y\in TM$. Similarly, if the minimal polynomial of $\varphi$ is of the form $P(x)=a_1x+a_3x^3+\cdots+a_{2N}x^{2N}+x^{2N+1}$, then $\varphi$ is minimal if and only if the restriction of \eqref{eq:35} to sections of $TM$ holds. It is worthwhile to note that while \eqref{eq:33} expresses minimality purely in terms of $f$ and its higher Nijenhuis torsions, the $a_k$ are the coefficients of the minimal polynomial of $\varphi$ which, as pointed out in Remark \ref{rem:29}, is in general not equal to the minimal polynomial of $f$. Hence, in the case of generalized polynomial structures defined by classical polynomial structures, minimality amounts to a novel relationship between the higher Nijenhuis torsions that, to the best of our knowledge, appears to be natural only from the vantage viewpoint of the generalized tangent bundle.
\end{example}

\begin{cor}
Let $\varphi$ be a non-zero semisimple polynomial structure. Then $\varphi$ is a generalized Nijenhuis operator if and and only if $\varphi$ is minimal with exactly two eigenvalues. 
\end{cor}

\begin{proof} One direction follows from Corollary \ref{cor:74} and Remark \ref{rem:54}. For the converse, let us assume that $\varphi$ is semisimple with exactly two eigenvalues. Then its minimal polynomial is of the form $P(x)=x^2-\lambda^2$ for some non-zero real number $\lambda$. Hence, by Example \ref{ex:47} we have $0=\mathcal M_\varphi= 2\mathcal  T_\varphi$ forcing $\varphi$ to be a generalized Nijenhuis operator. 
\end{proof}

\begin{rem}
Let $\varphi$ be a generalized polynomial structure and let $\{\x_\alpha\}\subseteq L_\lambda$, $\{\y_\beta\}\subseteq L_\mu$ be Jordan chains. In the notation of Remark \ref{rem:PolynomialAction}, using the identity $(u-uv-uw+vw)=(u-v)(u-w)$ as in \cite{Kosmann-Schwarzbach19} we immediately obtain
\begin{equation}\label{eq:30}
    \mathcal T_\varphi^{(m)}(\x_\alpha,\y_\beta)=\sum_{i,j=0}^{m} (-1)^{i+j}\binom{m}{i}\binom{m}{j}(\varphi-\lambda I)^{m-i}(\varphi-\mu I)^{m-j}\lbra \x_{\alpha-i},\y_{\beta-j}\rbra
\end{equation}
which generalizes Proposition 29 in \cite{TempestaTondo18a} to sections of $TM$. Because of this, many of the results of \cite{TempestaTondo18a} still hold if $TM$ is extended to $\T M$ and the Lie bracket is extended to the Dorfman bracket.
\end{rem}

\begin{example} Let $\varphi$ be a generalized polynomial structure with minimal polynomial of the form $P(x)=x^r Q(x)$ for some real polynomial $Q$ such that $Q(0)\neq 0$. It follows from \eqref{eq:30} that if $\x,\y\in {\rm Ker}(\varphi)$, then $\mathcal T_\varphi^{(m)}(\x,\y)=\varphi^{2m}\lbra \x,\y\rbra$. Replacing $\varphi$ with $\varphi^{r}$ we see that $L_0$ is closed under the Dorfman bracket if and only if the restriction of $\mathcal T_{\varphi^r}^{(m)}$ to $L_0$ vanishes. 
\end{example}

\begin{prop}\label{prop:72}
Let $\varphi$ be a generalized polynomial structure and let $\lambda\in \Sigma(\varphi)$. If there exists $m\ge 1$ such that $\mathcal T_\varphi^{(m)}(L_\lambda,L_\lambda)=0$, then $L_\lambda$ is closed under the Dorfman bracket. 
\end{prop}

\begin{proof}
Thanks to \eqref{eq:30}, the proof of the corresponding statement (Proposition 42) in \cite{TempestaTondo18a} can be adapted step-by-step without difficulty.
\end{proof}

\begin{cor}\label{cor:72}
Let $\varphi$ be a generalized polynomial structure that is neither invertible nor nilpotent. Then none of the higher Courant-Nijenhuis torsions of $\varphi$ vanishes identically. In particular, $\varphi$ is not a generalized Nijenhuis operator.
\end{cor}

\begin{proof}
Combine Proposition \ref{prop:72} and Corollary \ref{cor:42}.
\end{proof}

\begin{rem}
Proposition \ref{prop:72} can be generalized as follows (see Theorem 45 in \cite{TempestaTondo18a}). If $\lambda_1,\ldots,\lambda_k$ are eigenvalues of $\varphi$ such that the restriction of one of the higher Courant-Nijenhuis torsions of $\varphi$ to $L=L_{\lambda_1}\oplus \cdots\oplus L_{\lambda_k}$ vanishes, then $L$ is closed under the Dorfman bracket. In light of Remark \ref{rem:13}, we observe that this statement (and, in particular, Proposition \ref{prop:72}) is in general false if the Dorfman bracket is replaced by the (antisymmetric) Courant bracket.  
\end{rem}

\subsection{Bivariate Courant-Nijenhuis torsion}

\begin{mydef}
The {\it bivariate Courant-Nijenhuis torsion} of two commuting operators $\varphi'$ and $\varphi''$ acting on $\T M$ is defined as $\mathcal T_{\varphi'+\varphi''}-\mathcal T_{\varphi'}-\mathcal T_{\varphi''}$.
\end{mydef}

\begin{prop}\label{prop:93}
Let $\varphi'$, $\varphi''$ be generalized polynomial structures. Then $\varphi=\varphi'+\varphi''$ is minimal provided that the following conditions hold:
\begin{enumerate}[1)]
\item $\varphi'\circ \varphi''=0=\varphi''\circ \varphi'$;
\item $\varphi'$ and $\varphi''$ are both minimal;
\item $\varphi'$ and $\varphi''$ both annihilate the minimal polynomial of $\varphi$;
\item the bivariate Courant-Nijenhuis torsion of $\varphi'$ and $\varphi''$ vanishes.
\end{enumerate}
\end{prop}

\begin{proof} Let $P(x)=a_0+a_1x+\cdots+a_Nx^N$ be the minimal polynomial of $\varphi$.
We begin by noticing that 1) implies $[\varphi',\varphi'']=0$ and thus their bivariate Courant-Nijenhuis torsion as defined above is well defined. Adapting the notation of Remark \ref{rem:56} we denote by $\{u,v,w\}$, $\{u',v',w'\}$, and $\{u'',v'',w''\}$ generators of the polynomial actions on tensors defined, respectively, by $\varphi$, $\varphi'$, and $\varphi''$. In particular, $u=u'+u''$, $v=v'+v''$, and $w=w'+w''$. From 1), we see that $u'u''=v'v''=w'w''=0$. Using this we have
\begin{align}
    \mathcal T_{\varphi'+\varphi''}-\mathcal T_{\varphi'}-\mathcal T_{\varphi''}&=((u-v)(u-w)-(u'-v')(u'-w')-(u''-v'')(u''-w''))\cdot \zeta_{CD}\nonumber\\
    & = (u'-v'-w')(u''-v''-w'')\cdot \zeta_{CD}\label{eq:51}
\end{align}
and thus
\begin{equation}
    0=((u-v-w)^n-(u'-v'-w')^n-(u''-v''-w'')^n)\cdot \zeta_{CD}
\end{equation}
for all $n\ge 1$. Hence
\begin{equation*}
    \mathcal M_\varphi^{(n)}-\mathcal M_{\varphi'}^{(n)}-\mathcal M_{\varphi''}^{(n)}=\pm (u^n-v^n-w^n-(u')^n+(v')^n+(w')^n-(u'')^n+(v'')^n+(w'')^n)\cdot \zeta_{CD}.
\end{equation*}
By 3), we have $P(u)=P(u')=\cdots=P(w'')=0$. Hence
\begin{equation}\label{eq:46}
    \mathcal M_\varphi = \sum_{k=0}^n a_k \mathcal M_{\varphi'}^{(k)}+\sum_{k=0}^n a_k \mathcal M_{\varphi''}^{(k)}\,.
\end{equation}
Another consequence of 3) is that the minimal polynomial of $\varphi'$ divides $P(x)$. Together with 2), this implies
\begin{equation}\label{eq:47bis}
    \sum_{k=0}^N a_k \mathcal M^{(k)}_{\varphi'} = P(u'-v'-w')\cdot \zeta_{CD}=0\,.
\end{equation}
Finally, substituting \eqref{eq:47bis} and the analogous result for $\varphi''$ into \eqref{eq:46} shows that $\varphi$ is indeed minimal.
\end{proof}

\begin{example}\label{ex:102}
Let $\varphi$ be a generalized polynomial structure with minimal polynomial of the form $P(x)=x^r(x^2+c_0)(x^2+c_2)\cdots(x^2+c_s)$, for non-negative integers $r,s$ and arbitrary real numbers $c_1<c_2<\cdots<c_s$. Then the semisimple and the nilpotent parts of $\varphi$ satisfy conditions 1) and 3) in Proposition \ref{prop:93}. Hence $\varphi$ is minimal provided that $\varphi_s$ and $\varphi_n$ are both minimal with vanishing bivariate Courant-Nijenhuis torsion.
\end{example}

\begin{cor}\label{cor:103}
Let $\varphi$ be a generalized polynomial structure of degree 4 with exactly 3 distinct eigenvalues. Then $\varphi$ is minimal if and only if 
\begin{enumerate}[1)]
    \item the semisimple part $\varphi_s$ of $\varphi$ is a weak generalized Nijenhuis operator;
    \item the nilpotent part $\varphi_n$ of $\varphi$ is a generalized Nijenhuis operator;
    \item the bivariate Courant-Nijenhuis torsion of $\varphi_s$ and $\varphi_n$ vanishes.
\end{enumerate}
\end{cor}

\begin{proof} By assumption the minimal polynomial of $\varphi$ is of the form $P(x)=x^2(x^2+c)$ for some non-zero real number $c$.
In one direction this is a particular case of Example \ref{ex:102}. For the converse, assume that $\varphi$ is minimal. Proposition \ref{prop:87} then implies that $\varphi_s$ and $\varphi_n$ are both minimal. Thanks to the observations made in Example \ref{ex:87} and in Example \ref{ex:69}, this is enough to conclude that $\varphi_s$ is a weak generalized Nijenhuis operator and $\varphi_n$ is a generalized Nijenhuis operator. To prove 3), we employ the notation introduced in the proof of Proposition \ref{prop:93} with $\varphi'=\varphi_s$ and $\varphi''=\varphi_n$. Since, by the minimality of $\varphi_s$ and $\varphi_n$ we have $((u'-v'-w')^3+c(u'-v'-w'))\cdot \zeta_{CD}=0$ and $(u''-v''-w'')^2\cdot \zeta_{CD}=0$. Imposing the minimality of $\varphi$ we then obtain
\begin{equation}\label{eq:73}
    0=(u-v-w)^2((u-v-w)^2+c)\cdot \zeta_{CD}=-2c(u'-v'-w')(u''-v''-w'')\cdot \zeta_{CD}\,. 
\end{equation}
By \eqref{eq:51}, \eqref{eq:73} implies that the bivariate Courant-Nijenhuis torsion of $\varphi_s$ and $\varphi_n$ vanishes.
\end{proof}

\begin{cor}
There is a canonical bijection between the set of minimal generalized polynomial structures with minimal polynomial $P(x)=x^2(x^2+1)$ and the set of pairs $(\mathcal J_1,\mathcal J_2)$ consisting of a strongly integrable generalized F-structure $\mathcal J_1$ and a weak generalized tangent structure $\mathcal J_2$ such that 
\begin{enumerate}[1)]
    \item $\mathcal J_1\circ \mathcal J_2=0=\mathcal J_2\circ \mathcal J_1$
    \item the bivariate Courant-Nijenhuis torsion of $\mathcal J_1$ and $\mathcal J_2$ vanishes.
\end{enumerate}
\end{cor}

\begin{proof}
Given a generalized polynomial structure $\varphi$ with minimal polynomial $P(x)=x^2(x^2+1)$ we have that $\mathcal J_1=\varphi_s$ is a generalized F-structure and $\mathcal J_2=\varphi_n$ is a weak generalized almost tangent structure satisfying condition 1). Conversely if $\mathcal J_1$ and $\mathcal J_2$ are as above, then $\varphi=\mathcal J_1+\mathcal J_2$ is a generalized polynomial structure with minimal polynomial $P(x)=(x^2+1)^2$. The result then follows from Corollary \ref{cor:103}.
\end{proof}

\subsection{Decomposition of the de Rham operator}

In this subsection we show how any generalized polynomial structure $\varphi$ induces a canonical decomposition of the de Rham operator into components labeled by the spectrum of $\varphi$. The following theorem provides a characterization of minimality in terms of the compatibility of this decomposition with the adjoint action of $\varphi$.

\begin{theorem}\label{theorem:83}
Let $\varphi$ be a generalized polynomial structure with minimal polynomial
\begin{equation}
    P(x)=\prod_{\lambda\in \Sigma(\varphi)}(x-\lambda)^{m(\lambda)}\,.
\end{equation}
Then $\varphi$ is minimal if and only if there exists a decomposition of the de Rham operator
\begin{equation}\label{eq:38}
    d = \sum_{\lambda \in \Sigma(\varphi)} d_\lambda
\end{equation}
such that $(\ad_{\widetilde \varphi}-\lambda I)^{m(\lambda)}(d_\lambda)\in\T M$ for every $\lambda\in \Sigma(\varphi)$.
\end{theorem}
\begin{proof}
Assume such a decomposition exists. Since $(\ad_{\widetilde \varphi}-\lambda I)^{m(\lambda)}$ divides $P(\ad_{\widetilde\varphi})$, then $P(\ad_{\widetilde\varphi})(d_\lambda)\in \T M$ for every $\lambda\in \Sigma(\varphi)$. Hence $P(\ad_{\widetilde\varphi})(d)\in \T M$ and thus $\varphi$ is minimal. Conversely, assume that $\varphi$ is minimal.  let
\begin{equation}\label{eq:40}
    d_\lambda = \sum_{i=1}^{m(\lambda)}a_{\lambda,i}Q_{\lambda,i}(\ad_{\widetilde\varphi})(d)\,,
\end{equation}
where, $Q_{\lambda,i}$ is a polynomial defined by the relation $(x-\lambda)^i Q_{\lambda,i}(x)=P(x)$ for each $\lambda\in \Sigma(\varphi)$ and for each $i\in\{1,\ldots,m(\lambda)\}$. Then \eqref{eq:38} holds and, for each $\lambda\in \Sigma(\varphi)$
\begin{equation}\label{eq:41}
    (\ad_{\widetilde\varphi}-\lambda I)^{m(\lambda)}(d_\lambda)=\sum_{i=1}^{m(\lambda)}a_{\lambda,i}(\ad_{\widetilde\varphi}-\lambda I)^{m(\lambda)-i}P(\ad_{\widetilde\varphi})(d)\,.
\end{equation}
The minimality of $\varphi$ implies that $P(\ad_{\widetilde\varphi})(d)$ and thus \eqref{eq:41} is a section of the generalized tangent bundle.
\end{proof}

\begin{rem}
A priori the definition of $d_\lambda$ given in \eqref{eq:41} depends on the particular lift of $\varphi$ to $\mathcal E_M$ obtained by imposing condition 3) in Proposition \ref{prop:9}. On the other hand, if $\varphi$ is minimal then choosing a different lift $\widetilde \varphi +g$ for some function $g\in \Omega_M^0$ would result in a a different decomposition $d=\sum_\lambda d_\lambda'$ such that $d_\lambda-d_\lambda'$ is a section of $\T M\otimes \mathbb C$ annihilated by $(\ad_{\widetilde \varphi}-\lambda I)^{m(\lambda)}$ i.e.\ a section of $L_\lambda$. As $\lambda$ runs through $\Sigma(\varphi),$, the sum of these sections is $\sum_\lambda (d_\lambda-d'_\lambda)=0$ and yet,  by Theorem \ref{theorem:40} no cancellation can occur between sections labelled by different eigenvalues. Hence $d_\lambda=d'_\lambda$ for all $\lambda\in \Sigma(\varphi)$ and the decomposition $\eqref{eq:40}$ induced by a minimal generalized polynomial structure does not depend on its lift to $\mathcal E_M$. This is consistent with the observation made in Remark \ref{rem:68} that the minimal torsion and thus the notion of minimality does not depend on such a lift.
\end{rem}

\begin{example}\label{ex:103}
Let $\varphi$ be a generalized almost complex structure. Then $a_{\pm \sqrt{-1},1}=\mp\frac{\sqrt{-1}}{2}$ and $Q_{\pm \sqrt{-1},1}(x)=x\pm \sqrt{-1}$. Substituting into \eqref{eq:40} we obtain
\begin{equation}
    d_{\sqrt{-1}}=\frac{1}{2}(d-\sqrt{-1}[\varphi,d])= \overline \partial 
\end{equation}
and, similarly, $d_{-\sqrt{-1}}=\partial$. 
\end{example}

\begin{rem}
Since \eqref{eq:40} expresses $d_\lambda$ as the image of $d$ under a polynomial in $\ad_{\widetilde \varphi}$, then $\lbra \x,\y\rbra_{d_\lambda}\in \T M\otimes \mathbb C$ for all $\x,\y\in \T M\otimes \mathbb C$.
\end{rem}

\begin{cor}\label{cor:104}
Let $\varphi$ be a minimal generalized polynomial structure and let $d_\lambda$ be as in \eqref{eq:40}. Then $\lbra L_\mu, L_\nu \rbra_{d_\lambda}\subseteq L_{\lambda+\mu+\nu}$ for all $\lambda,\mu,\nu\in \Sigma(\varphi)$.
\end{cor}

\begin{proof}
Adapting the notation of Remark \ref{rem:PolynomialAction}, let $\zeta_{CD,\lambda}$ denote the local linear operator representing the derived bracket with respect to $d_\lambda$. Since $\varphi$ is minimal, by Theorem \ref{theorem:83} we have $(u-v-w-\lambda)^{m(\lambda)}\cdot \zeta_{CD,\lambda}=0$. Recall the shift operator $\mathbf S$ introduced in the proof of Corollary \ref{cor:88}. Multiplying both sides by $(u-v-w-\lambda)^{m(\mu)+m(\nu)}$ and setting $m=m(\lambda)+m(\mu)+m(\nu)$ we obtain
\begin{align*}
    0&=(\varphi-(\lambda+\mu+\nu)I+\mathbf S)^m \lbra \x_\alpha,\y_\beta\rbra_{d_\lambda}\\ 
    &= \sum_{k=0}^{m(\mu)+m(\nu)}\binom{m}{k}(\varphi-(\lambda+\mu+\nu)I)^{m-k}\mathbf S^k\lbra \x_\alpha,\y_\beta\rbra_{d_\lambda}
\end{align*}
for all Jordan chains $\{\x_\alpha\}\subseteq {\rm Ker}((\varphi-\mu I)^p)$ and $\{\y_\beta\}\subseteq {\rm Ker}((\varphi-\nu I)^q)$. If $p=q=1$ this reduces to $(\varphi-(\lambda+\mu+\nu)I)^m\lbra \x_\alpha,\y_\beta \rbra_{d_\lambda}=0$ which implies $\lbra {\rm Ker}(\varphi-\mu I),{\rm Ker}(\varphi-\nu I)\rbra_{d_\lambda}\subseteq L_{\lambda+\mu+\nu}$. The result then follows by induction on $p+q$.
\end{proof}

\begin{rem}
As a consequence of Corollary \ref{cor:104}, if $\varphi$ is a minimal generalized polynomial structure, then
\begin{equation}\label{eq:57}
    \lbra \x ,\y \rbra = \sum_{\lambda\in \Sigma(\varphi)} \lbra \x,\y\rbra_{d_\lambda} \subseteq L_\mu+L_\nu + \sum_{\lambda\in \Sigma(\varphi)\setminus\{-\mu,-\nu\}} L_{\mu+\nu+\lambda}
\end{equation}
for all $\x\in L_\mu$, $\y\in L_\nu$.
\end{rem}

\subsection{The non-resonance condition}

The following definition is motivated by \eqref{eq:47}.

\begin{mydef}
Let $\varphi$ be a generalized polynomial structure with minimal polynomial $P(x)$. We say that $\varphi$ is {\it non-resonant} if $P(\lambda+\mu+\nu)\neq 0$ for any $\lambda,\mu,\nu \in \Sigma(\varphi)$ such that $(\lambda+\mu)(\lambda+\nu)(\mu+\nu)\neq 0$.
\end{mydef}

\begin{rem}
Non-resonance is a generic condition requiring that no eigenvalue can be written as a non-trivial (i.e.\ when two of the summands cancel each other) sum of three eigenvalues. 
\end{rem}

\begin{example}
Every generalized polynomial structure with at most three eigenvalues is non-resonant. In particular, generalized polynomial structure of degree at most 3 are always non-resonant.
\end{example}

\begin{example} Let $\varphi$ be a generalized polynomial structure with at most five distinct eigenvalues (in particular, this occurs if $\varphi$ has degree at most five). Then $\varphi$ is non-resonant if and only if $\lambda\notin \{2\mu, 3\mu\}$ for all $\lambda,\mu\in \Sigma(\varphi)$.
\end{example}

\begin{theorem}\label{theorem:93}
Let $\varphi$ be a minimal non-resonant generalized polynomial structure. Then the semisimple part of $\varphi$ is a weak generalized Nijenhuis operator.
\end{theorem}

\begin{proof}
The non-resonance assumption guarantees that if $\mu+\nu\neq 0$, then \eqref{eq:57} reduces to $\lbra L_\mu,L_\nu\rbra\subseteq L_\mu+L_\nu$. Hence the result follows from Theorem \ref{theorem:56}.
\end{proof}

\begin{prop}\label{prop:116}
Let $\varphi$ be a non-resonant generalized polynomial structure. Then its semisimple part $\varphi_s$ is minimal if and only if $\varphi_s$ is a weak generalized Nijenhuis operator.
\end{prop}

\begin{proof}
In one direction this statement 2) in Proposition \ref{prop:88} specialized to the non-resonant case. The converse follows from Theorem \ref{theorem:93} applied to $\varphi_s$ (which is equal to its own semisimple part), taking into account that $\varphi$ and $\varphi_s$ share the same generalized eigenbundles. 
\end{proof}

\begin{example}\label{ex:117}
Let $\varphi$ a generalized polynomial structure of with minimal polynomial $P(x)=(x^2+c_1)^{m_1}(x^2+c_2)^{m_2}$ for some positive integers $m_1$, $m_2$ and real numbers $c_1$, $c_2$ such that $c_1c_2(c_1-c_2)(c_1-4c_2)(4c_1-c_2)(c_1-9c_2)(9c_1-c_2)\neq 0$. Then $\varphi_s$ is minimal if and only it is a weak generalized Nijenhuis operator. In particular, a generic generalized polynomial structure of degree $4$ is minimal if and only if it is a weak generalized Nijenhuis operator. 
\end{example}

\begin{example}
To illustrate the fact that the non-resonance assumption in Theorem \ref{theorem:93} and in Proposition \ref{prop:116} cannot be completely removed, consider the following example. Let $M$ be the cartesian product of the 3-dimensional Heisenberg nilmanifold with $S^1$. Then $\mathbb T M$ can be globally trivialized by vector fields $\e_1,\ldots,\e_4$ and dual 1-forms $\e^1,\ldots,\e^4$ with Dorfman brackets encoded by the single (up to skew-symmetry) constraint $T(\e_1,\e_2,\e^3)=1$. Let $\lambda$, $\mu$ be real numbers and let $\varphi$ be the unique skew-symmetric endomorphism of $\T M$ such that $\varphi(\e_1)=\lambda \e_1$, $\varphi(\e_2)=\lambda \e_2$, $\varphi(\e_3)=\mu \e_3$ and $\varphi(\e_4)=\mu \e_4$. Then the minimal polynomial of $\varphi$ is $P(x)=(x^2-\lambda^2)(x^2-\mu^2)$. Let us assume $\lambda\neq \mu$ so that $\varphi$ is semisimple and $L_\lambda$ is not closed under the Dorfman bracket. Let us further assume that $\lambda\neq 0$ so that by Theorem \ref{theorem:56} we know that $\varphi$ cannot be a weak generalized Nijenhuis operator.  From \eqref{eq:47} we then have that $\varphi$ is minimal if and only if 
\begin{equation}
    0=\mathcal C_\varphi(\e_1,\e_2,\e^3)=P(2\lambda-\mu)=4\lambda(3\lambda-\mu)(\lambda-\mu)^2
\end{equation}
if and only if $3\lambda=\mu$, which as shown in Example \ref{ex:117} is a case in which non-resonance fails to hold. Hence $3\lambda=\mu\neq 0$ implies that $\varphi$ is minimal and semisimple but not a weak generalized Nijenhuis operator.    
\end{example}

\begin{rem}
Let $\varphi$ be a non-resonant minimal generalized polynomial structure so that, by Theorem \ref{theorem:93}, $\varphi_s$ is a weak generalized Nijenhuis operator. In principle, we have two decompositions of the de Rham operator, namely \eqref{eq:38} and the decomposition $d=\sum_\lambda \delta_\lambda$ whose existence is guaranteed by Theorem \ref{theorem:56}. By comparing Remark \ref{rem:61} and Example \ref{ex:103}, we see that in the case of generalized complex structures one has $\delta_{\pm\sqrt{-1}}=d_{\mp\sqrt{-1}}$. For the general case, we note that specializing Corollary \ref{cor:104} by imposing the non-resonance condition, we can repeat the proof of 2)$\Rightarrow$ 3) in Theorem \ref{theorem:56} with $\delta_\lambda$ replaced by $d_{-\lambda}$. Thus, $d_{\pm\lambda}$ has degree $\mp 1$ with respect to the $(L_\lambda, L_{-\lambda})$-grading and degree 0 with respect to the $(L_\mu,L_{-\mu})$-grading for every $\mu\in \Sigma(\varphi)\setminus\{0,\pm \lambda\}$. In particular, we conclude that if $\varphi$ is minimal and non-resonant, then $d_\lambda=\delta_{-\lambda}$ for every $\lambda\in \Sigma(\varphi)\setminus\{0\}$. \end{rem}

\section{Examples: invariant polynomial structures on Lie groups}\label{sec:6}

In this section, $G$ will denote a real Lie group and $\mathfrak{g}$ will be its Lie algebra. The invariant polynomial structures on $G$ coincide with all skew-symmetric endomorphisms of the space $D_1=D(\mathfrak{g})=\mathfrak{g}\oplus\mathfrak{g}^*$ (the \emph{Drinfeld double} of $\mathfrak{g}$), identified with the space of all left-invariant sections of $\T G$.  
Note that $D_1$ is itself a Lie algebra with respect to the Dorfman bracket.
In the following examples we will consider elements of the vector space $D_3=D_1^{\otimes 3}$, associated to the Lie algebra structure and an invariant polynomial structure $\varphi:D_1\rightarrow D_1$.
Note that the tautological inner product defines an isomorphism $D_1^*\rightarrow D_1, \ \omega\mapsto \widehat{\omega}$
such that $\omega(\x)=2\langle\widehat{\omega},\x\rangle$ for all $\x\in D_1$. Via tensor product, we obtain isomorphisms
$D_1^{*\otimes h}\otimes D_1^{\otimes k}\rightarrow D_1^{\otimes(h+k)}, {\mathcal T} \mapsto \widehat{{\mathcal T}}$.
In particular, any invariant polynomial structure $\varphi$ produces the elements
$\widehat{\mathcal{C}}_{\varphi},\ \widehat{\mathcal{T}}^{(n)}_{\varphi},\ \widehat{\mathcal{S}}^{(n)}_{\varphi}\in D_3 $.
In the following examples we calculate the elements  $\varphi_1^{i_1}\varphi_2^{i_2}\varphi_3^{i_3}\cdot \widehat{\mathcal{T}}^{(0)}_{\varphi}$, where the action was defined in Remark \ref{rem:PolynomialAction}: in particular, we will represent these elements with respect to a Jordan basis and adopt the following abbreviation:
$$\x\y\z:=\x\otimes\y\otimes\z.$$
Note that the polynomial action of $\varphi$ is graded skew-symmetric with respect to the  following non-degenerate inner product:
$$\langle\x_1\x_2\x_3, \y_1\y_2\y_3\rangle=\prod_{i=1}^3\langle\x_i,\y_i\rangle$$
In our examples, $\{v_i\}_{i=1}^n$ will denote a fixed basis of $\mathfrak{g}$, with dual basis given by $\{\alpha_i\}_{i=1}^n$.
Moreover, in our description of the tensors it will be convenient to work with the complexification  of $\varphi$ and to use to use a basis built upon the block decomposition of $\varphi$.
More specifically, we will choose bases of $D_1\otimes\C$ of the form $$\left\{\b_j^{V,l}:j,V,l\right\}$$ where the indices $V,l$ and $j$ are described as follows:
\begin{itemize}
    \item each index $V$ indicates a distinct indecomposable complex block;
    \item for a fixed block $V$ of degree $k$ we have $1\leq j\leq k$ and the vectors $$\left\{\b_k^{V,l}:l\right\}$$ form a basis of the semisimple part $W$ as described in Section \ref{decompos}. Also, these vectors will be chosen to be real if the eigenvalues of $V$ are real;
    \item for fixed $V$ (of degree $k$) and $l$, we have $\varphi_n\left(\b_j^{V,l}\right)=\b_{j-1}^{V,l}$ for all $j=2,3,\dots k$ while $\varphi_n\left(\b_1^{V,l}\right)=0$, that is we have the Jordan chain
    $$\b_k^{V,l}\longmapsto \b_{k-1}^{V,l}\longmapsto \b_{k-2}^{V,l}\longmapsto \b_{k-3}^{V,l}\longmapsto\dots\longmapsto\b_1^{V,l}\longmapsto 0.$$
    (We will simplify the indices $V,l$ to lighten the notation).
\end{itemize}
\subsection{The Heisenberg group.}
 Let $\mathfrak{g}=\mathfrak{n}(3)$, the 3-dimensional Heisenberg Lie algebra, which is nilpotent and has structure equations
$$[v_1,v_2]=v_3, \quad [v_1,v_3]=[v_2,v_3]=0.$$
Therefore, 
$$\widehat{\mathcal{T}}^{(0)}_{\varphi}=\alpha_1\alpha_2v_3\pm \text{perm.}$$
where \lq\lq $\pm\ \text{perm.}$\rq\rq refers to the signed sum of terms where the three factors are permuted.

\subsubsection{Example.} Define the endomorphism $\varphi:D_1\rightarrow D_1$ as follows:
$$\varphi(v_1)=v_2,\quad \varphi(v_2)=v_3+\alpha_3,\quad \varphi(v_3)=-\alpha_2,$$
$$\varphi(\alpha_1)=0,\quad \varphi(\alpha_2)=-\alpha_1, \quad \varphi(\alpha_3)=-\alpha_2$$
By construction, $\varphi$ is skewsymmetric, it has
minimal polynomial $P(x)=x^5$, and two Jordan chains 
$$v_1\longmapsto v_2\longmapsto v_3+\alpha_3\longmapsto -2\alpha_2\longmapsto 2\alpha_1\longmapsto 0$$
$$v_3-\alpha_3\longmapsto 0$$
which span two real blocks, of types $\Delta_5^+(0)$ and  $\Delta_1^-(0)$ respectively.
Defining
$$\b_1=2\alpha_1,\quad \b_2=-2\alpha_2,\quad \b_3=v_3+\alpha_3,\quad \b_4=v_2,\quad \b_5=v_1,\quad \b_1'=v_3-\alpha_3$$
we get
\begin{equation}
-8\widehat{\mathcal{T}}^{(0)}_{\varphi}=\b_1\b_2\b_3+\b_1\b_2\b_1'\pm \text{perm.}
\end{equation}
Upon inspection, we see that the polynomials that do not annihilate the summands have total degree at most $3$. in particular, $\widehat{\mathcal{C}}_{\varphi}=P(\varphi_1+\varphi_2+\varphi_3)\cdot \widehat{\mathcal{T}}^{(0)}_{\varphi}=0$, i.e. $\varphi$ is minimal. For the same reason, $\widehat{\mathcal{T}}^{(k)}_{\varphi}=0$ for $k\geq 2$. Moreover,
\begin{align*}
-8\widehat{\mathcal T}^{(1)}_{\varphi}=&(\varphi_3^2+\varphi_1\varphi_2+\varphi_1\varphi_3+\varphi_2\varphi_3)\cdot (-8\widehat{\mathcal T}^{(0)}_{\varphi})\\
=&(\b_1\b_2\b_1 -\b_2\b_1\b_1)+(\b_1\b_2\b_1-\b_2\b_1\b_1)\\
& +(-\b_1\b_1\b_2+\b_2\b_1\b_1)+(\b_1\b_1\b_2-\b_1\b_2\b_1)\\
=&-\b_2\b_1\b_1+\b_1\b_2\b_1\,,
\end{align*}
while 
\begin{equation}
-8\widehat{{\mathcal S}}^{(1)}_{\varphi}=(-\varphi_1-\varphi_2)\cdot (-\b_2\b_1\b_1+\b_1\b_2\b_1)=\b_1\b_1\b_1-\b_1\b_1\b_1=0\,,
\end{equation}
and also $\widehat{{\mathcal S}}^{(n)}_{\varphi}=0$ for all $n\geq 1$.

\subsubsection{Example}
Let now $$\varphi(v_1)=0,\quad \varphi(v_2)=-v_1,\quad \varphi(v_3)=-v_2,$$
$$\varphi(\alpha_1)=-\alpha_2,\quad \varphi(\alpha_2)=v_3+\alpha_3, \quad \varphi(\alpha_3)=-v_2$$
As in the previous example, $\varphi$ is skew-symmetric, is has minimal polynomial $P(x)=x^5$ and real block decomposition
$\Delta_5^+(0)\oplus\Delta_1^-(0)$. Its 
Jordan chains are
$$\alpha_1\longmapsto \alpha_2\longmapsto v_3+\alpha_3\longmapsto -2v_2\longmapsto 2v_1\longmapsto 0$$
$$v_3-\alpha_3\longmapsto 0$$
and block decomposition of same type as in the previous example. Defining
$$\b_1=2v_1,\quad \b_2=-2v_2,\quad \b_3=v_3+\alpha_3,\quad \b_4=\alpha_2,\quad \b_5=\alpha_1,\quad \b_1'=v_3-\alpha_3$$
we get
\begin{equation}
2\widehat{\mathcal{T}}^{(0)}=\b_5\b_4\b_3+\b_5\b_4\b_1'\pm \text{perm.}
\end{equation}
In this case, upon inspection we see that the polynomials that do not annihilate the summands have total degree at most $9$, so that $\widehat{\mathcal T}^{(k)}_{\varphi}=0$ and  $\widehat{\mathcal S}^{(k)}_{\varphi}=0$ for $k\geq 5$.  Now we prove that $\widehat{\mathcal{C}}_{\varphi}\neq 0$,
by showing that
$\langle\widehat{\mathcal{C}}_{\varphi}, \x\y\z\rangle\neq 0$ for some $\x\y\z$.
Note first that
\begin{equation}\label{eq:82}
\langle\widehat{\mathcal{C}}_{\varphi}, \x\y\z\rangle=\langle P(\varphi_1+\varphi_2+\varphi_3)\cdot\widehat{\mathcal{T}}^{(0)}_{\varphi}, \x\y\z\rangle=-\langle \widehat{\mathcal{T}}^{(0)}_{\varphi}, P(\varphi_1+\varphi_2+\varphi_3)\cdot\x\y\z\rangle
\end{equation}
Setting $\x\y\z=\b_1'\b_5\b_3$ in \eqref{eq:82}, we get
\begin{equation}
P(\varphi_1+\varphi_2+\varphi_3)\cdot \b_1'\b_5\b_3=5\b_1'\b_1\b_2+10\b_1'\b_2\b_1
\end{equation}
and 
\begin{align*}
\langle2\widehat{\mathcal{T}}^{(0)}_{\varphi},5\b_1'\b_1\b_2+10\b_1'\b_2\b_1\rangle=&\langle\b_1'\b_5\b_4-\b_1'\b_4\b_5,5\b_1'\b_1\b_2+10\b_1'\b_2\b_1\rangle\\
=&5\langle\b_1'\b_5\b_4,\b_1'\b_1\b_2\rangle-10\langle\b_1'\b_4\b_5,\b_1'\b_2\b_1\rangle\\
=&-5\,,
\end{align*}
so that $\langle\widehat{\mathcal{C}}_{\varphi}, \b_1'\b_5\b_3\rangle \neq 0$. Therefore, $\varphi$ is not minimal. \\
Now, we are going to prove that $\widehat{{\mathcal T}}^{(4)}_{\varphi}\neq 0$, by showing that $\langle\widehat{{\mathcal T}}^{(4)}_{\varphi}, \b_5\b_4\b_5\rangle\neq 0$. In fact,
\begin{align*}
\langle 2\widehat{\mathcal T}^{(4)}_{\varphi}, \b_5\b_4\b_5\rangle=&\langle2\widehat{\mathcal T}^{(0)}_{\varphi}, (\varphi_1+\varphi_3)^4(\varphi_2+\varphi_3)^4\cdot\b_5\b_4\b_5\rangle\\
=&\sum_{i,j}\binom{4}{i}\binom{4}{j}\langle2\widehat{\mathcal T}^{(0)}_{\varphi}, \b_{1+i}\b_{j}\b_{5-i-j}\rangle\\
=& 6\langle\b_5\b_4\b_3,\b_1\b_2\b_3\rangle-4\langle\b_5\b_3\b_4,\b_1\b_3\b_2\rangle -16\langle\b_4\b_5\b_3,\b_2\b_1\b_3\rangle\\
&+16\langle\b_4\b_3\b_5,\b_2\b_3\b_1\rangle+24\langle\b_3\b_5\b_4,\b_3\b_1\b_2\rangle-36\langle\b_3\b_4\b_5,\b_3\b_2\b_1\rangle\\
=&10\,.
\end{align*}
However, a similar calculation shows that $\widehat{\mathcal{S}}^{(4)}_{\varphi}=0$.

\subsubsection{Example}
Define $$\varphi(v_1)=v_2-\alpha_2, \quad \varphi(v_2)=-v_1+\alpha_1,\quad \varphi(v_3)=0$$
$$\varphi(\alpha_1)=\alpha_2\quad \varphi(\alpha_2)=-\alpha_1,\quad\varphi(\alpha_3)=0$$
In this case, $\varphi$ has minimal polynomial $P(x)=x(x^2+1)^2$, block decomposition $\Delta_2^+(\sqrt{-1},-\sqrt{-1})\oplus \Delta_1^0(0,0)$
and Jordan chains
$$v_1\longmapsto -\alpha_2\longmapsto 0,\qquad v_2\longmapsto \alpha_1\longmapsto 0$$
$$v_3\longmapsto 0 \qquad \alpha_3\longmapsto 0$$
Define
$$\b_1^1=-\alpha_2-\sqrt{-1}\alpha_1,\quad \b_2^1=v_1-\sqrt{-1}v_2,\quad \b_i^2=\overline{\b_i^1},\quad \b_1^3=v_3, \quad \b_1^4=\alpha_3 $$
Then, 
\begin{equation}
2\sqrt{-1}\widehat{\mathcal{T}}^{(0)}_{\varphi}=\b_1^1\b_1^2\b_1^3\pm \text{perm.}
\end{equation}
Therefore,
$\widehat{\mathcal{C}}_{\varphi}, \widehat{\mathcal{T}}^{(n)}_{\varphi}, \widehat{\mathcal{S}}^{(n)}_{\varphi}$ do not depend on the nilpotent part of $\varphi$, and
$\varphi_1^{i_1}\varphi_2^{i_2}\varphi_3^{i_3}\cdot \widehat{\mathcal{T}}^{(0)}_{\varphi}=0$ if $i_1i_2i_3\neq 0$.
Moreover, if $i_1i_2\neq 0
$,
\begin{equation}
2\varphi_1^{i_1}\varphi_2^{i_2}\cdot\widehat{\mathcal{T}}^{(0)}_{\varphi}=\left({\sqrt{-1}}\right)^{i_1+i_2-1}\left((-1)^{i_2}\b_1^1\b_1^2\b_1^3
-(-1)^{i_1}\b_1^2\b_1^1\b_1^3
\right)
\end{equation}
and analogously, if $i_1i_3\neq 0$,
\begin{equation}
2\varphi_1^{i_1}\varphi_3^{i_3}\cdot\widehat{\mathcal{T}}^{(0)}_{\varphi}=\left({\sqrt{-1}}\right)^{i_1+i_3-1}\left(-(-1)^{i_3}\b^1_1\b_1^3\b_1^2+(-1)^{i_1}\b_1^2\b_1^3\b_1^1\right)
\end{equation}
and if $i_2i_3\neq 0$,
\begin{equation}
2\varphi_2^{i_2}\varphi_3^{i_3}\cdot\widehat{\mathcal{T}}^{(0)}_{\varphi}=\left({\sqrt{-1}}\right)^{i_2+i_3-1}\left((-1)^{i_3}\b_1^3\b_1^1\b_1^2-(-1)^{i_2}\b_1^3\b_1^2\b_1^1\right)\,.
\end{equation}
Finally, for $i\neq 0$,
\begin{align*}
2\varphi_1^{i}\cdot\widehat{\mathcal{T}}^{(0)}_{\varphi}=&\left({\sqrt{-1}}\right)^{i-1}\left(\b_1^1\b_1^2\b_1^3-\b^1_1\b_1^3\b_1^2-(-1)^i\b_1^2\b_1^1\b_1^3+(-1)^i\b_1^2\b_1^3\b_1^1\right)\\
2\varphi_2^{i}\cdot\widehat{\mathcal{T}}^{(0)}_{\varphi}=&\left({\sqrt{-1}}\right)^{i-1}\left((-1)^i\b_1^1\b_1^2\b_1^3-\b_1^2\b_1^1\b_1^3+\b_1^3\b_1^1\b_1^2-(-1)^i\b_1^3\b_1^2\b_1^1\right)\\
2\varphi_3^{i}\cdot\widehat{\mathcal{T}}^{(0)}_{\varphi}=&\left({\sqrt{-1}}\right)^{i-1}\left(-(-1)^i\b^1_1\b_1^3\b_1^2+\b_1^2\b_1^3\b_1^1+(-1)^i\b_1^3\b_1^1\b_1^2-\b_1^3\b_1^2\b_1^1\right)\,.
\end{align*}
In particular, $(\varphi_1+\varphi_2+\varphi_3)\cdot\widehat{\mathcal{T}}^{(0)}_{\varphi}=0$, whence $\widehat{\mathcal{C}}_{\varphi}=0$. Also,
\begin{align*}
2\widehat{\mathcal{T}}^{(n)}_{\varphi}=&2(\varphi_1+\varphi_3)^n(\varphi_2+\varphi_3)^n\cdot \widehat{\mathcal{T}}^{(0)}_{\varphi}\\
=&2\varphi_1^n\varphi_2^n\cdot \widehat{\mathcal{T}}^{(0)}_{\varphi}\\
=&-\sqrt{-1}\left(\b_1^1\b_1^2\b_1^3
-\b_1^2\b_1^1\b_1^3\right)
\end{align*}
which is not vanishing for all integers $n>0$. This calculation also implies $\widehat{\mathcal{S}}^{(n)}_{\varphi}=0$ for all integers $n>0$. 
\subsubsection{Example}

$$\varphi(v_1)=v_2, \quad \varphi(v_2)=-v_1,\quad \varphi(v_3)=0$$
$$\varphi(\alpha_1)=\alpha_2-v_2, \quad \varphi(\alpha_2)=-\alpha_1+v_1,\quad \varphi(\alpha_3)=0$$
$$\alpha_1\longmapsto -v_2\longmapsto 0, \quad \alpha_2\longmapsto v_1 \longmapsto 0$$
$$v_3\longmapsto 0, \quad \alpha_3\longmapsto 0$$
$$\Delta_2^+(\sqrt{-1},-\sqrt{-1})\oplus \Delta_1^0(0,0)$$
Let
$$\b_1^1=-v_2-\sqrt{-1}v_1, \quad \b_2^1=\alpha_1-\sqrt{-1}\alpha_2,\quad \b_i^2=\overline{\b_i^1},\quad \b_1^3=v_3, \quad \b_1^4=\alpha_3$$
Hence
\begin{equation}
    2\sqrt{-1}\widehat{\mathcal{T}}^{(0)}_{\varphi}=\b_2^1\b_2^2\b_1^3\pm \text{perm.}
\end{equation}
so that, as before, $\varphi_1\varphi_2\varphi_3\cdot \widehat{\mathcal{T}}^{(0)}_{\varphi}=0$. In order to prove that $\widehat{{\mathcal C}}_{\varphi}\neq 0$, it is sufficient to show that $\langle \widehat{{\mathcal C}}_{\varphi}, \b_1^1\b_2^2\b_1^4\rangle\neq 0$.
First of all, note that
\begin{equation}
\langle \widehat{{\mathcal C}}_{\varphi}, \b_1^1\b_2^2\b_1^4\rangle=-\langle \widehat{{\mathcal T}}_{\varphi}^{(0)}, R(\varphi_1,\varphi_2)\cdot \b_1^1\b_2^2\b_1^4\rangle\,,
\end{equation}
where
\begin{equation}
R(\varphi_1,\varphi_2)=5(\varphi_1^4\varphi_2+\varphi_1\varphi_2^4)+10(\varphi_1^3\varphi^2_2+\varphi_1^2\varphi_2^3)+6(\varphi_1^2\varphi_2+\varphi_1\varphi_2^2)\,.
\end{equation}
Now, for all positive integers $a,b$ with $a+b$ odd,
\begin{equation}
(\varphi_1^a\varphi_2^b+\varphi_1^b\varphi_2^a)\cdot \b_1^1\b_2^2\b_1^4=(a-b)\left(\sqrt{-1}\right)^{a+b-1}(-1)^b\b_1^1\b_1^2\b_1^4
\end{equation}
whence
\begin{equation}
R(\varphi_1,\varphi_2)\cdot\b_1^1\b_2^2\b_1^4=\b_1^1\b_1^2\b_1^4
\end{equation}
and
\begin{align*}\langle \b_2^1\b_2^2\b_1^3-\b_2^1\b_1^3\b_2^2-&\b_2^2\b_2^1\b_1^3+\b_2^2\b_1^3\b_2^1+\b_1^3\b_2^1\b_2^2-\b_1^3\b_2^2\b_2^1  \    ,\ \b_1^1\b_1^2\b_1^4    \rangle=\\
=&\langle -\b_2^2\b_2^1\b_1^3\      ,\ \b_1^1\b_1^2\b_1^4    \rangle\neq 0,
\end{align*}
which implies $\widehat{{\mathcal C}}_{\varphi}\neq 0$. Similar calculations show that $\widehat{\mathcal{T}}^{(n)}_{\varphi}\neq 0 $ and $\widehat{\mathcal{S}}^{(n)}_{\varphi}\neq 0 $ for $n\geq 1$.

\subsection{A four-dimensional nilpotent group} Let $\mathfrak{g}$ be the four-dimensional nilpotent Lie algebra with structure equations
\begin{equation}
[v_1,v_2]=v_3, \quad [v_2,v_4]=-v_1
\end{equation}
and $[v_i,v_j]=0$ if $\{i,j\}\neq \{1,2\}$ or $\{2,4\}$. In this case,
\begin{equation}
\widehat{\mathcal{T}}^{(0)}_{\varphi}=\alpha_1\alpha_2v_3-\alpha_2\alpha_4v_1\pm \text{perm.}
\end{equation}

\subsubsection{Example} The endomorphism $\varphi: D_1\rightarrow D_1$ defined by
$$\varphi(v_1)=0, \quad \varphi(v_2)=v_4, \quad \varphi(v_3)=v_1, \quad \varphi(v_4)=v_3,$$
$$\varphi(\alpha_1)=-\alpha_3,\quad \varphi(\alpha_2)=0, \quad \varphi(\alpha_3)=-\alpha_4, \quad \varphi(\alpha_4)=-\alpha_2$$
is skew-symmetric, has single indecomposable real block of type $\Delta_4^0(0,0)$
and minimal polynomial $P(x)=x^4$. Moreover, its Jordan chains are
$$v_2\longmapsto v_4\longmapsto v_3\longmapsto v_1\longmapsto 0,$$
$$-\alpha_1\longmapsto\alpha_3\longmapsto-\alpha_4\longmapsto\alpha_2\longmapsto 0.$$
As before, setting
$$\b_1^1=v_1, \quad \b_2^1=v_3, \quad \b_3^1=v_4,\quad \b_4^1=v_2$$
$$\b_1^2=\alpha_2, \quad \b_2^2=-\alpha_4, \quad \b_3^2=\alpha_3,\quad \b_4^2=-\alpha_1$$
we get
\begin{equation}\widehat{\mathcal{T}}^{(0)}_{\varphi}=-\b_4^2\b_1^2\b_2^1+\b_1^2\b_2^2\b_1^1 \pm \text{perm.}
\end{equation}
and 
\begin{equation}
    \widehat{\mathcal{C}}_{\varphi}=(\varphi_1+\varphi_3)^4\cdot(-\b_4^2\b_1^2\b_2^1)\pm \text{perm.}=-4\b_1^2\b_1^2\b_1^1\pm \text{perm.}=0\,,
\end{equation}
so that $\varphi$ is minimal. Moreover, for $n\geq 1$,
\begin{equation}
    \widehat{\mathcal T}^{(n)}_{\varphi}=(\varphi_2\varphi_3+\varphi_3^2)^n(\b_1^2\b_4^2\b_2^2-\b_1^2\b_2^2\b_4^2)+(\varphi_1\varphi_3+\varphi_3^2)^n(-\b_4^2\b_1^2\b_2^2+\b_2^2\b_1^2\b_4^2)
\end{equation}
whence
\begin{equation}
\widehat{\mathcal T}^{(1)}_{\varphi}=\b_1^2\b_3^2\b_1^2-\b_1^2\b_2^1\b_3^2-\b_1^2\b_2^2\b_2^2-\b_3^2\b_1^2\b_1^2+\b_1^2\b_1^2\b_3^2+\b_2^2\b_1^2\b_2^2
\end{equation}
and $\widehat{\mathcal T}^{(n)}_{\varphi}=0$ for $n\geq 2$. Finally,

\begin{equation}\widehat{\mathcal S}^{(1)}_{\varphi}=\b_1^2\b_2^2\b_1^2-\b_1^2\b_1^1\b_3^2-\b_2^2\b_1^2\b_1^2\end{equation}
and $\widehat{\mathcal S}^{(n)}_{\varphi}=0$ for $n\geq 2$.

\vskip.1in\noindent
\address{Marco Aldi\\
Department of Mathematics and Applied Mathematics\\
Virginia Commonwealth University\\
Richmond, VA 23284, USA\\
\email{maldi2@vcu.edu}}

\vskip.1in\noindent
\address{Daniele Grandini\\
Department of Mathematics and Economics\\
Virginia State University\\
Petersburg, VA 23806, USA\\
\email{dgrandini@vsu.edu}}

\end{document}